\documentclass[a4paper,11pt]{article}
\usepackage{amsmath,amssymb,amsthm} 
\usepackage{graphics} 

\usepackage{psfrag}
\usepackage{epsfig}
\usepackage{color}
\newtheorem{theorem}{Theorem}[section]
\newtheorem{lemma}[theorem]{Lemma}
\newtheorem{proposition}[theorem]{Proposition}

\theoremstyle{definition}

\newtheorem{remark}[theorem]{Remark}

\newcommand{\R}{\mathbf{R}}

\newcommand{\eps}{\varepsilon}
\newcommand{\oo}{\infty}
\newcommand{\nb}{\nabla}
\newcommand{\ov}{\overline}

\newcommand{\dw}{\downarrow}

\newcommand{\Om}{\Omega}

\newcommand{\ds}{\displaystyle}

\newcommand{\be}{\begin{equation}}
\newcommand{\ee}{\end{equation}}
\newcommand{\oee}{\ov{\eta}_\eps}
\DeclareMathOperator{\supp}{supp}

\newcommand\wtilde{\widetilde}

\newcommand\Gxi{\mathcal{G}_\xi}

\newcommand\lt{\left}
\newcommand\rt{\right}
\newcommand\F{\mathcal{F}}
\newcommand\jj{J}
\newcommand\J{\mathcal{J}}

\newcommand\rhob{{\ov\rho}}

\def\void{\varnothing}
\def\Div{\textup{div}\,}
\def\HH{\mathcal{H}}
\def\les{\lesssim}
\def\ges{\gtrsim}
\newcommand{\bsigma}{{\sigma_{\lambda,K}}}

\title{Phase segregation for binary mixtures of Bose-Einstein Condensates}
\author{M. Goldman\footnote{ CNRS, UMR 7598, Laboratoire Jacques-Louis Lions, F-75005, Paris, France, email: goldman@math.univ-paris-diderot.fr} \and B. Merlet\footnote{Laboratoire P. Painlev\'e, CNRS UMR 8524, Universit\'e Lille 1,
F-59655 Villeneuve d'Ascq Cedex, France, email: benoit.merlet@math.univ-lille1.fr}}
\date{}
\begin{document}
\maketitle
\begin{abstract}
 We study the strong segregation limit for mixtures of Bose-Einstein condensates modeled by a Gross-Pitaievskii functional. Our first main result is that in presence of a trapping potential, for different intracomponent
 strengths, the Thomas-Fermi limit is sufficient to determine the shape of the minimizers. Our second main result is that for asymptotically equal intracomponent strengths,  one needs to go to the next order. The relevant limit is a weighted isoperimetric problem. 
 We then study the minimizers of this limit problem, proving radial symmetry or symmetry breaking for different values of the parameters. 
 We finally show that in the absence of a confining potential, even for non-equal intracomponent strengths, one needs to study a related isoperimetric problem to gain information about the shape of the minimizers.   
\end{abstract}

\section{Introduction}

In this paper, we investigate the asymptotic behavior as $\eps$ goes to zero of the Gross-Pitaievskii functional

\be\label{funcintrobis}
F_\eps(\eta)\,=\, \eps \int_{\R^2}  |\nb \eta_1|^2 + |\nb \eta_2|^2\, + \frac1\eps\int_{\R^2}\frac{1}2 \eta_1^4   + \frac{g}2 \eta_2^4  +  K \eta_1^2\eta_2^2 +  (\eta_1^2 +\eta_2^2)V,\ee
under the mass constraint
\be\label{constrainteta}
\int_{\R^2} \eta_1^2=\alpha_1 \qquad \textrm{and} \int_{\R^2} \eta_2^2=\alpha_2.
\ee
This functional arises in the study of two component Bose-Einstein condensates. It has been widely studied, both in the physical and mathematical literature (see \cite{GolRoyo} and the references therein or the book \cite{AfLivre}).
The potential $V$ is a trapping potential. {For simplicity, we only consider here} the harmonic potential $V=|x|^2$. The constant $g$, measures the asymmetry between the intracomponent repulsive strengths of each component 
and $K$ represents the intercomponent repulsive strength. Without loss of generality, we will take here $g\ge 1$.
The case $K<\sqrt{g}$, where mixing of the two condensates occurs  has recently been well understood in  \cite{AftNorSour}. On the other hand,  the case $K>\sqrt{g}$, 
where it is expected both experimentally \cite{McCarronETall,PaJILA}, numerically \cite{MaAf,OhSten} and theoretically \cite{KaTsuUe, Aochui,Timmermans,barankov,BVS} that segregation occurs, 
is maybe not yet so well understood. In the symmetric case $g=1$, it has been proved in \cite{AftRoyo,GolRoyo} that, as expected from the physics literature,
the Thomas-Fermi limit i.e. the limit of $\eps F_\eps$ only imposes segregation but does not give any information about the actual shape of the minimizers. To gain such information, one has to go to the next order
since $F_\eps- \min_{\eta} F_{\eps}(\eta)$ converges in the sense of $\Gamma-$convergence \cite{braidesbook} to a weighted isoperimetric problem. In this paper, we focus on the asymmetric case $g>1$ and show that the situation is radically different. 

Let us introduce the Thomas-Fermi energy:
\begin{equation}\label{thomasfermiintro}
  E ( \rho) = \  
  \int_{\R^2} \lt[\frac{1}2 \rho_1^2   + \frac{g}2 \rho_2^2  +  K \rho_1\rho_2 + (\rho_1+\rho_2)V\rt].
\end{equation}

Our first main theorem is the following:

\begin{theorem}\label{maintheointro}
For every $\alpha_1,\alpha_2, g, K>0$, with $K\ge\sqrt{g}>1$, there exists a unique minimizer $\rho^{0}=(\rho^0_{1},\rho^{0}_{2})$ of~\eqref{thomasfermiintro} under the volume constraints
\be\label{ctrtmasse}\int_{\R^2} \rho_1=\alpha_1 \qquad \textrm{and} \qquad \int_{\R^2} \rho_2=\alpha_2.\ee
This minimizer $\rho^{0}$ is radially symmetric (see~Lemma~\ref{lemcpct1} for explicit formulas for $\rho^0$ and $E_0=E(\rho^0)$).  \\
Moreover, we have the following stability result: there exists $C>0$ (which depends only on $\alpha_1,\alpha_2$ and $g$) such that if $\rho$ satisfies the constraints~\eqref{ctrtmasse} then\footnote{we denote by $\|\cdot\|_p$ the $L^p$ norm}
\be\label{quantiestim}
\|\rho-\rho^0\|_1\le C\sqrt{E(\rho)-E_0}.
\ee
\end{theorem}
Let us point out that the radial symmetry of the minimizer of the Thomas-Fermi energy was not completely expected (see \cite{BVS,MaAf} or the discussion in \cite[Sec. 1.3.4]{AftRoyo}). 
As a consequence of this stability result, we prove that minimizers of $F_\eps$ converge to $(\sqrt{\rho^0_1},\sqrt{\rho^0_2})$.
\begin{theorem}
\label {thmGPtoTF}
 There exists $C> 0$ such that for $\eps\in(0,1]$, any minimizer $\eta^\eps$ of $F_\eps$ under the constraints \eqref{constrainteta}
satisfies
 \be\label{quantconv1}
 \lt\|\eta^\eps-\lt(\sqrt{\rho_1^0},\sqrt{\rho_2^0}\rt)\rt\|_2\le C \eps^{1/4}.
 \ee
\end{theorem}
This theorem establishes that in the non-symmetric case,  the Thomas-Fermi limit already provides full information on the limiting behavior of the minimizers of~\eqref{funcintrobis}.
 It is quite surprising that even without using isoperimetric effects, we are able to obtain strong convergence of the minimizers in the form of~\eqref{quantconv1}.    Let us point out that the idea of using stability inequalities such as~\eqref{quantiestim} to get (quantitative) convergence results is far from new (see for instance \cite{CarFig,BelGolZwi}).
One crucial point which explains the difference between the asymmetric case and the symmetric one is that here, there is a gap between the two Thomas-Fermi profiles $\rho_1^0$ and $\rho_2^0$ in the sense that there exists $r_0>0$ such that $\supp \rho_0^1\subset \ov{B(0,r_0)}$ and $\supp\rho_0^2\subset \R^2\setminus B(0,r_0)$ with  
\[
 \inf_{r<r_0} \rho^0_1(r)\ >\  \sup_{r> r_0} \rho^0_2(r).
\]
\\

We then study the crossover case where $g=1+\eps \xi$ for some $\xi>0$. Let $\overline{\eta}_\eps$ be the minimizer of 
\[G_\eps(\eta)= \eps \int_{\R^2} |\nabla \eta|^2 +\frac{1}{\eps} \int_{\R^2} \frac{1}{2} \eta^4 + V \eta^2\]
under the mass constraint 
\[\int_{\R^2} \eta^2 =\overline{\alpha},\]
where $\overline{\alpha}=\alpha_1+\alpha_2$.
It is well known \cite{IgMil, KaSour} that $\overline{\eta}_\eps^2$ converges when $\eps$ goes to zero to the Thomas-Fermi profile
\[\rhob= (R^2-V)_+\]
where $R$ is chosen such that $\int_{\R^2} \rhob=\overline{\alpha}$. Building on results obtained in the case $g=1$ \cite{GolRoyo,AftRoyo}, we get
\begin{theorem}\label{gammaintro}
 For $K>1$, the functional $(F_\eps-G_\eps(\overline{\eta}_\eps))$ $\Gamma-$ converges  as $\eps$ goes to zero for the strong $L^1$ topology  to
 \[\Gxi(u_1,u_2)=\begin{cases}
                    \ds \sigma_{K}\int_{\partial E} \rhob^{3/2}+ \xi \int_{E^c} \rhob^2 \qquad & \ds \textrm{if } u_1= \rhob \chi_E, \ u_2= \rhob \chi_{E^c} \textrm{ and } \int_E \rhob=\alpha_1,\\
                   +\infty \qquad & \textrm{otherwise},
                  \end{cases}\]
where $\sigma_{K}>0$ is defined by the one dimensional optimal transition problem
\[\sigma_{K}=\inf \left\{ \int_{\R} | \eta'_1|^2 + | \eta'_2|^2+  \frac{1}{2}\left(\eta_1^2+\eta_2^2-1\right)^2 +(K-1) \eta_1^2\eta_2^2
\ :\ \lim_{-\infty} \eta_1=0, \ \lim_{+\infty} \eta_1=1\right\}.
\]
\end{theorem}
We study the minimizers of the limiting functional $\Gxi$ and prove our second main result.
\begin{theorem}\label{theoxiintro}
 The following holds:
 \begin{itemize}
  \item there exists $\alpha_0\in (0,\ov\alpha/2]$ such that for every $\alpha_1\in (\alpha_0, \overline{\alpha}-\alpha_0)$ there exists $\xi^1_{\alpha_1}$
  such that the minimizer of $\Gxi$ is not radially symmetric for $\xi\le \xi^1_{\alpha_1}$,
  \item for every $\alpha_1\in (0, \overline{\alpha})$, there exists $\xi^2_{\alpha_1}$ such that the minimizer of $\Gxi$ is the centered ball for $\xi\ge \xi_{\alpha_1}^2$.
 \end{itemize}

\end{theorem}
The regime $g=1+\eps\xi$ corresponds to the numerical simulation of \cite{MaAf}. In that paper, the observed numerical results (droplets) fit with the first point of Theorem \ref{theoxiintro} (see in particular \cite[Fig 1.c, Fig. 3.a]{MaAf}).
The first part of Theorem \ref{theoxiintro} is a consequence of a symmetry breaking result from \cite{AftRoyo}. The second part follows from a combination of two results.
The first is a stability result for the  functional $\int_{E^c} \rhob^2$:
\begin{proposition}
 For every $\alpha\in (0,\overline{\alpha})$, there exists $C=C(\alpha)>0$ such that for every measurable set  $E\subset\R^2$ with $\int_E \rhob= \alpha$, we have,
\[
\int_{E^c} \rhob^2-\int_{B_{r}^c} \rhob^2\ \geq\ C \left(\int_{{E\Delta B_{r}}} \rhob \right)^2
\]
where $r$ is such that $\int_{B_r} \rhob=\alpha$.
\end{proposition}
The second is an estimate on the potential instability of the ball for the weighted isoperimetric problem.

\begin{proposition}
  For every $\alpha\in (0,\overline{\alpha})$, there exists $c=c(\alpha)>0$ such that for every set  $E\subset \R^2$  with locally finite perimeter and with $\int_E \rhob= \alpha$, there holds
\be\label{stabisoperintro} \int_{\partial E} \rhob^{3/2}-\int_{\partial B_r} \rhob^{3/2}\ \geq\ -c \left(\int_{E\Delta B_r} \rhob \right)^2,\ee
where $r$ is such that $\int_{B_r} \rhob=\alpha$.
\end{proposition}
 The rigidity result given by Theorem \ref{theoxiintro} is similar in spirit to several rigidity results obtained for variants of isoperimetric problems (see \cite{KnuMu, FigMag, GolNovRuf} for instance).
 However, the  peculiar aspect here is that this rigidity does not come from the isoperimetric term but rather from the volume term. Nevertheless, the proof of~\eqref{stabisoperintro}
 bounding the {\it instability} of the ball follows the strategy of \cite{CicLeo} (see also \cite{AcFuMo,FigMag}) to prove quantitative {\it stability} estimates for isoperimetric problems. The idea is to show first the desired inequality for nearly spherical sets
 following ideas of \cite{fuglede} and then use the regularity theory for minimal surfaces to reduce oneself to this situation. 
 To the best of our knowledge, it is the first time that this strategy has been implemented to control the {\it instability} of the ball. Let us also point out that one of the ingredients in our proof is the following isoperimetric inequality:
 \begin{lemma}
  There exists a constant $c>0$ such that for every measurable set $E\subset\R^2$ satisfying $\int_{E}\rhob\leq \alpha/2$, there holds
  \[\int_{\partial E} \rhob^{3/2} \ge c \left(\int_E \rhob\right)^{5/6}.\]
 \end{lemma}
In recent years, there has been an increasing interest in studying isoperimetric problems with densities (see \cite{MorPra, FigMag,DePFraPra} for instance). However, most of these authors consider either problems where the same density is used for weighting the volume and the perimeter or weights which are increasing at infinity.\\

In the last part of the paper we come back to the situation $g>1$ but consider an infinitely stiff trapping potential. That is, we assume that $V$ is equal to zero inside some given open set $\Omega$ and is infinite outside. This is somehow the setting which is considered in \cite{BVS,SchayIndcrit}. 
In this case, it is easier to work with slightly different parameters. After a new rescaling and some simple algebraic manipulations (see Section \ref{notrap}), the problem can be seen to be equivalent to minimizing
\be\label{funcintroter}\jj_\eps(\eta)= \eps\int_{\Omega} |\nabla \eta_1|^2 +\lambda^2 |\nabla \eta_2|^2 + \frac{1}{\eps}\int_{\Omega} \frac{1}{2}( \eta_1^2+ \eta_2^2 -1)^2 + (K-1) \eta_1^2\eta_2^2\ee
with the volume constraint
\[\int_{\Omega} \eta_1^2 =\alpha_1 \qquad \textrm{and} \qquad \int_{\Omega} \eta_2^2 =\alpha_2,\]
for some $\lambda\le 1$, $K>1$ and $\alpha_1,\alpha_2\geq 0$ such that $\alpha_1+\alpha_2=|\Omega|$. The main difficulty in studying the $\Gamma-$convergence of~\eqref{funcintroter} to a sharp limit model is to obtain strong compactness for sequences of bounded energy. 
In the symmetric case $\lambda=1$, which corresponds to $g=1$ in~\eqref{funcintrobis}, one can follow the strategy of \cite{AftRoyo,GolRoyo} and use a nonlinear sigma model representation \cite{KaTsuUe}  to rewrite the problem in terms of an amplitude and a phase. In these unknowns, the functional takes a form similar to the celebrated Ambrosio-Tortorelli functional \cite{AmbTort} from which one can get compactness.
In the non-symmetric case, the methods of \cite{AftRoyo,GolRoyo} do not apply. One has to find a different approach. Inspired by recent work on type-I superconductors \cite{CoGoOtSe}, we prove instead that for every $\lambda \le 1$, the energy~\eqref{funcintroter} directly controls a Modica-Mortola \cite{modica} type energy where $\eta_1$ and $\eta_2$ are decoupled (see~\eqref{mainestim}). 
As a consequence, we recover the compactness of sequences of bounded energy. Besides extending results of \cite{AftRoyo,GolRoyo} to the non-symmetric case, we believe that this more direct approach also has an intrinsic interest. 
In fact, it seems more natural and gives a better understanding of how the interaction between the two condensates via the term $\int_{\Omega} \eta_1^2 \eta_2^2$ gives to the energy a structure of a double well potential. Using then classical arguments from $\Gamma-$convergence such as the slicing method \cite{braides}, 
we can prove our last main theorem:

\begin{theorem}\label{gamma2intro}
 When $\eps\to 0$, $\jj_\eps$ $\Gamma-$converges for the strong $L^1$ topology to 
 \[\J(\eta_1,\eta_2)=\begin{cases}
                                \sigma_{\lambda,K} P(E,\Omega) & \textrm{if } \eta_1=\chi_E=1-\eta_2 \textrm{ and } |E|=\alpha_1\\
                                +\infty & \textrm{otherwise},
                               \end{cases}\]
where, for a set of finite perimeter $E$ (see \cite{AmbFusPa,Maggibook}), $P(E,\Omega)$ denotes the relative perimeter of $E$ inside $\Omega$ and where $\sigma_{\lambda,K}>0$ is defined by the one dimensional optimal transition problem
\[\sigma_{\lambda,K}=\inf \left\{ \int_{\R} | \eta'_1|^2 +\lambda^2 | \eta'_2|^2+  \frac{1}{2}\left(\eta_1^2+\eta_2^2-1\right)^2 +(K-1) \eta_1^2\eta_2^2
: \lim_{-\infty} \eta_1=0, \ \lim_{+\infty} \eta_1=1\right\}.
\]
 
\end{theorem}

At last, in the spirit of what was done in \cite{GolRoyo} in the case $\lambda=1$ (see also the recent paper \cite{AftSour}), we study the asymptotic behavior of $\sigma_{\lambda,K}$ as $K$ goes to one (mixing) or $K$ goes to
infinity (strong segregation) and recover good parts of what is expected from the physics literature \cite{barankov,BVS,Aochui,Timmermans}, namely
\[
\lim_{K\to 1} \frac{\sigma_{\lambda,K}}{\sqrt{K-1}}\,=\, \dfrac23\dfrac{1-\lambda^3}{1-\lambda^2},\qquad\quad \sigma_{\lambda,K}-\sigma_{\lambda,\infty}\,\stackrel{\frac{\lambda^2}{K}\downarrow0}{\sim}\, -\lt(\dfrac1{K^{{1/2}}} +\dfrac{\lambda^{1/2}}{K^{1/4}}\rt),
\]
see Propositions~\ref{propK1} and~\ref{propKoo}.

Before closing this introduction, let us point out that most of the results, in particular Theorem \ref{maintheointro}, Theorem \ref{gammaintro} and Theorem \ref{gamma2intro} can be easily generalized to arbitrary dimension and arbitrary radially symmetric strictly increasing confining potentials.\\

The structure of the paper is the following.  In Section \ref{difg}, we study the case of different intracomponent strengths in the presence of a confining potential. We then study in Section \ref{crossover} 
the crossover case and finally in Section \ref{notrap}, we investigate the case of non-equal intracomponent strengths in the absence of confining potential and in a bounded domain.

\section*{Notation}

For $x\in \R^2$ and $r>0$, we denote by $B_r(x)$ the ball of radius $r$ centered at $x$ and simply write $B_r$ when $x=0$. 
Given a set $E\subset \R^2$, we let $\chi_E$ be the characteristic function of $E$.  For any integer $k$, we denote by $\mathcal{H}^k$ the $k-$dimensional Hausdorff measure  The letters, $c, C$ denote universal constants which can vary from line to line. We also make use of the usual $o$ and $O$ notation.  The symbols $\sim$, $\ges$, $\les$ indicate estimates that hold up to a positive constant. For instance, $f\les g$ denotes the existence of a constant $C>0$ such that $f\le Cg$.
 Throughout the paper, with a small abuse of language, we call sequence a family $(u_\eps)$ of functions 
labeled by a continuous parameter $\eps\in (0,1]$. A subsequence 
of $(u_\eps)$ is any sequence $(u_{\eps_k})$ such that $\eps_k
\to 0$ as $k \to +\infty$. For $1\le p\le+ \infty$ and a function $f$, we denote the $L^p$ norm of $f$ by $\|f\|_p$. In the sequel, when it is clear from the context, we omit to indicate the integrating measure. 
In particular, all integrals involving boundaries are with respect to $\mathcal{H}^1$ and all integrals involving sets are with respect to the Lebesgue measure.\\

\section{The case of different intracomponent strengths}\label{difg}

\subsection{The Thomas-Fermi profile}
Let us consider the Thomas-Fermi approximation  and prove Theorem~\ref{maintheointro}. Let $\alpha=(\alpha_1,\,\alpha_2)\in (0,+\oo)^2$. We denote by $X$ the set of pairs of measurable functions $(\rho_1,\rho_2)\ :\ \R^2\to \R_+^2$ and by  $X_\alpha$ the subset of pairs $\rho\in X$ satisfying 
\begin{equation*}
  \int_{\R^2} \rho_j \  = \  \alpha_j \quad \mbox{  for }
  \  j = 1, 2.
\end{equation*}
For $\rho\in X$ we define the energy
\begin{equation*}
  E ( \rho) = 
  \int_{\R^2} \frac{1}2 \rho_1^2   + \frac{g}2 \rho_2^2  +  K \rho_1\rho_2 + V(\rho_1+\rho_2).
\end{equation*}
 As stated in Theorem~\ref{maintheointro}, we study the minimization of $E$ in the class $X_\alpha$: we prove the existence of a unique minimizer for this problem and a stability result.

We start by studying the minimization problem in the following subsets : given $r>0$, let  
\[
X^{0,r}_\alpha\ =\ \big\{\rho\in X_\alpha\ :\ \supp(\rho_1)\subset \ov{B_r},\ \supp(\rho_2)\subset \R^2\setminus B_r \
\big\}.
\]
and 
\begin{equation}\label{X0alpha}
X^{0}_\alpha\ =\ \bigcup_{r>0}X^{0,r}_\alpha.
\end{equation}
\begin{figure}[ht]
\psfrag{a}{$\rho^0_1$}
\psfrag{b}{$\rho^0_2$}
\psfrag{c}{$-r_0$}
\psfrag{d}{$r_0$}
\psfrag{e}{$\sigma_+$}
\psfrag{f}{$\sigma_-$}
\psfrag{t}{$T$}
\begin{center}
\includegraphics[scale=1.1]{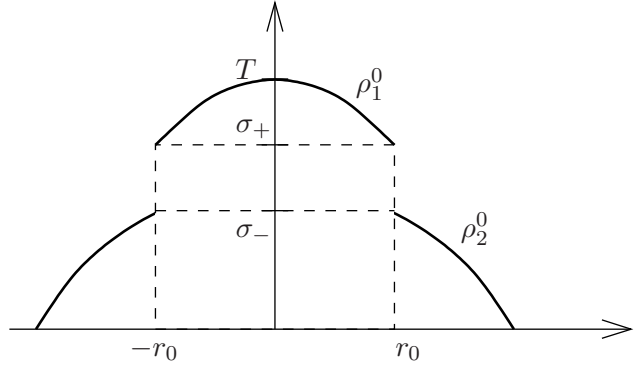}
\end{center}
\caption{\label{rho0} Minimizer of the Thomas-Fermi energy.}
\end{figure}
\begin{lemma}

\label{lemcpct1} 
For $K\geq \sqrt{g}>1$ and $r>0$, $E$ admits a unique minimizer $\rho^{0,r}=(\rho^{0,r}_1,\rho^{0,r}_2)$ in $X_\alpha^{0,r}$.  Let us set 
\[ 
\ov{r}=\min(r,r_1)\quad\mbox{with }r_1=\lt( \frac{2\alpha_1}\pi \rt)^{1/4}\quad\mbox{and then }\quad R_1^2\,=\,\frac{\ov{r}^2}{2}+\frac{r_1^4}{2\ov{r}^2} ,\quad R_2^2\,=\,r^2+\left(\frac{2 g\alpha_2}{\pi}\right)^{1/2}.
\] 
The unique minimizer is given by 
\be\label{formrhor} 
\rho^{0,r}_{1}(x)\ =\ \begin{cases} (R_1^2-|x|^2)_{+} & \mbox{if }|x|<r,\\
~\qquad0&\mbox{if } |x|>r,
\end{cases}
\qquad\ 
\rho^{0,r}_{2}(x)\ =\ \begin{cases} ~\qquad0 & \mbox{if }|x|<r,\\
 \frac{1}{g}(R^2_2-|x|^2)_{+} & \mbox{if }|x|>r.
\end{cases}\quad~
\ee

Optimizing in $r>0$, the minimum of $E$ in $X^0_\alpha$ is reached at $\rho^0=\rho^{0,r_0}$ (see Figure \ref{rho0}) with
\be\label{r0}
r_0\, =\, r_1\lt( \sqrt{1 + \frac{\alpha_2}{\alpha_1} }  - \sqrt{\frac{\alpha_2}{\alpha_1}} \rt)^{1/2}\, \in(0,r_1).
\ee
In particular, there is a positive gap between $\rho^0_2$ and $\rho^0_1$ at the frontier $\partial B_{r_0}$:
\be\label{gap}
\sigma_+\,=\,\inf_{B_{r_0}} \rho^0_1 \, =\,  \sqrt{g} \sup \rho^0_2\, >\, \sup  \rho^0_2\,=\,\sigma_-.
\ee
The minimal energy is 
\be\label{E0}
E_0\, =\, E(\rho^0)\ =\ \frac{2}{3} \sqrt{\frac{2 }\pi} \lt( (\alpha_1+\alpha_2)^{3/2} +   \left(\sqrt{g}-1 \right) \alpha_2^{3/2}  \rt).
\ee
\end{lemma}

\begin{proof}~\\
Let $r>0$. The optimization of $E$ in $X^{0,r}_\alpha$ splits into two independent optimization problems. From the associated Euler-Lagrange equations, 
we see that the minimizers have the form given by~\eqref{formrhor} with $R_1$, $R_2$ fixed by the conditions $\int_{\R^2} \rho_j=\alpha_j$ for $j=1,2$. 
\\

Let us now study the variation of $r\mapsto E^r=\min_{X_\alpha^{0,r}}E= E_1(\rho_1^{0,r})+E_2(\rho_2^{0,r})$. Using the notation $t=(r/r_1)^2$, we obtain by direct computation,
\[
E^r\,=\, f(t) +\dfrac{2}{3}\left(\frac{2 g \alpha_2^3}{\pi}\right)^{1/2},
\]
with
\[f(t)\, =\, 
\begin{cases}
\dfrac\pi{24} r_1^6\left( 6t + \dfrac{3}{t}- t^3\right) + \alpha_2 r_1^2 t \,  & \mbox{ for } t<1,\\ \\
 \dfrac\pi{3} r_1^6+ \alpha_2 r_1^2 t &\mbox{ for }t\geq 1.
\end{cases}
\]
We have $f\in C^2(0,+\oo)$ with $f''>0$ in $(0,1)$,  $f(t)\to+\oo$ as $t\dw0$ and $f'=f'(1)>0$ on $(1,+\oo)$. Therefore, $f$ admits a unique minimiser $t_0=(r_0/r_1)^2\in(0,1)$.\\
Eventually,  after some algebraic computations, we get that $r_0$ is given by~\eqref{r0}, we also obtain~\eqref{gap} and~\eqref{E0}.
\end{proof}

We are now ready to prove Theorem~\ref{maintheointro}.

\begin{proof}[Proof of Theorem~\ref{maintheointro}]
Let $\rho\in X_\alpha$. We want to show that $E(\rho)\geq E(\rho^0)$. Since $\rho^0$ and $E(\rho^0)$ do not depend on $K$ and $E(\rho)$ is a non decreasing function of $K$, we only consider the worst case $K=\sqrt{g}$. Then, we write,
\[
\rho=\rho^0+\delta\!\rho,\qquad \mbox{with $\delta\!\rho=(\delta\!\rho_1,\delta\!\rho_2)$.}
\]
The energy {expands} as
\[
E(\rho)\,=\, E(\rho^0) + L_1(\delta\!\rho_1) + L_2(\delta\!\rho_2)+Q(\delta\!\rho),
\]
with 
\[
L_1(\delta\!\rho_1)\, =\, \int_{\R^2} (\rho^0_1+\sqrt{g}\rho^0_2 +|x|^2 )\delta\!\rho_1,\qquad L_2(\delta\!\rho_2)\,=\, \int_{\R^2} (g \rho^0_2+\sqrt{g} \rho^0_1 +|x|^2 )\delta\!\rho_2,
\]
\[
Q(\delta\!\rho)\,=\, \dfrac12 \int_{\R^2} (\delta\!\rho_1 +\sqrt{g}\delta\!\rho_2)^2.
\]
This last term is obviously {non-negative}. Let us study the terms $L_1(\delta\!\rho_1)$ and $L_2(\delta\!\rho_2)$.  Let $r_0$, $R_1$, $R_2$ be as in the definition of $\rho^0$ in Lemma~\ref{lemcpct1}. We denote  by $A$ the annulus $B_{R_2}\setminus\ov{B_{r_0}}$ and by $U$ the {exterior} domain $\R^2\setminus B_{R_2}$.  Using the definition of $\rho^0_1$ and $\rho^0_2$, we have
\[
L_1(\delta\!\rho_1)\, =\, R_1^2 \int_{B_{r_0}} \delta\!\rho_1 + \int_{A} \lt[\dfrac{1}{\sqrt{g}} R_2^2  + \lt(1-\dfrac{1}{\sqrt{g}} \rt) |x|^2\rt]\delta\!\rho_1+ \int_{U} |x|^2\delta\!\rho_1.
\]
Since $\int_{\R^2} \delta\!\rho_1=0$, we have $\int_{B_{r_0}} \delta\!\rho_1=-\int_{A\cup U}\delta\!\rho_1$. Using this, we get,
\[
L_1(\delta\!\rho_1)\,=\, \int_{A\cup U} c_1(|x|) \delta\!\rho_1,\qquad\mbox{with }c_1(r)\,=\,\begin{cases} \dfrac{1}{\sqrt{g}} (R_2^2 -r^2) - (R_1^2 -r^2)&\mbox{ for } r_0<r<R_2,\\
\qquad r^2-R_1^2&\mbox{ for } r\geq R_2.
\end{cases}
\]
We notice that $c_1$ is continuous on $[r_0,+\oo)$ and of the form $a+br^2$ with $b>0$ on the two intervals $(r_0,R_2)$ and $(R_2,+\oo)$. Moreover  $c_1(r_0)=0$. 
Hence $c_1>0$ in $(r_0,+\oo)$, $c_1(r_0)=0$, $c_1'(r_0)>0$ and $c_1(r)\stackrel{r\uparrow\oo}{\sim} r^2$. {Now recall that $\rho_1^0+\delta \!\rho_1\geq 0$ and since $\rho_1^0\equiv 0$ in $A\cup U$, we have $\delta\!\rho_1\geq 0$ in $A\cup U$. Hence, $L_1(\delta\!\rho_1)\geq 0$.}

Similarly, 
\[
L_2(\delta\!\rho_2)\,=\,\int_{B_{r_0}\cup U} c_2(|x|) \delta\!\rho_2,\qquad\mbox{with }c_2(r)\,=\,\begin{cases}  \sqrt{g}(R_1^2 -r^2) - (R_2^2 -r^2)&\mbox{ for } r<r_0,\\
\qquad r^2-R_2^2&\mbox{ for } r>R_2.
\end{cases}
\] 
In the unbounded interval  $(R_2,+\oo)$, we have $c_2(r)>0$ with $c_2(R_2)=0$, $c_2'(R_2)>0$ and $c_2(r)\stackrel{r\uparrow\oo}{\sim} r^2$. In the interior interval $(0,r_0)$, we have $c_2>0$, with $c_2(r_0)=0$, $c_2'(r)<0$.
Hence $c_2>0$ in $[0,r_0)$. Since $\delta\!\rho_2\geq 0$ in $B_{r_0}\cup U$, we have $L_2(\delta\!\rho_2)\geq0$. 
\medskip

To sum up, we have establish that for $\rho=\rho^0+\delta\!\rho\in X_\alpha$, there holds
\[
E(\rho)\,=\, E(\rho^0) +H(\delta\!\rho)\, \geq\, E(\rho^0),
\]
where we have set 
\[
H(\delta\!\rho)\, =\,  \int_{A\cup U} c_1(|x|) \delta\!\rho_1\, +  \int_{B_{r_0}\cup U} c_2(|x|) \delta\!\rho_2 \,+\dfrac12\int_{\R^2} (\delta\!\rho_1+\sqrt{g}\delta\!\rho_2)^2.
\]
\bigskip

Now, for proving the stability estimate~\eqref{quantiestim}, we have to show that there exists $C\geq 0$ such that for every function $\delta\!\rho=(\delta\!\rho_1,\delta\!\rho_2)$ such that $\rho^0+\delta\!\rho\in X_\alpha$, there holds
\begin{equation}\label{preuvestab0}
 \|\delta\!\rho\|_{L^1}^2 \, \leq\, C H(\delta\!\rho).
 \end{equation}

We prove~\eqref{preuvestab0} in several steps. 
First let us set $R=R_2+1$, $V=\R^2\setminus\ov{B_R}$ and $\tilde A=\ov{A\cup U}\setminus V$ so that $B_{r_0}\cup\tilde{A}\cup V=\R^2$. We estimate successively $\|\delta\!\rho\|_{L^1(V)}$, $\|\delta\!\rho\|_{L^1(\tilde A)}$ and $\|\delta\!\rho\|_{L^1(B_{r_0})}$.\\
{\it Step 1. Estimating $\|\delta\!\rho\|_{L^1(V)}$.}  In the set $V$, we have $\delta\!\rho_1,\delta\!\rho_2\geq 0$. Using the fact that $c_1$ and $c_2$ are positive and increasing on $[R,+\oo)$, we deduce 
\[
\int_{V}|\delta\!\rho_1|+|\delta\!\rho_2|\, \leq\, \dfrac1{\min(c_1(R),c_2(R))} \int_V c_1(|x|)\delta\!\rho_1 + c_2(|x|)\delta\!\rho_2\, \leq\, C H(\delta\!\rho).
\]
Since the condition $\rho^0+\delta\!\rho \in X_\alpha$ implies $\int_{\R^2} |\delta\!\rho_j|\leq 2\alpha_j$ for $j=1,2$, we get 
\begin{equation}\label{preuvestab1}
 \|\delta\!\rho\|_{L^1(V)}^2 \, \leq\, C H(\delta\!\rho).
 \end{equation}
 
 {\it Step 2. Estimating $\|\delta\!\rho\|_{L^1(\tilde A)}$.} Notice that $\delta\!\rho_1\geq 0$ in $\tilde A$. We split the annulus $\tilde A$ into three subsets $\tilde A=\tilde A_1 \cup \tilde A_2 \cup \tilde A_3$ with
 \begin{eqnarray*}
 \tilde A_1&=& \lt\{x\in \tilde{A}\, :\, \delta\!\rho_2(x)\geq 0 \rt\},\\
 \tilde A_2 &=& \lt\{x\in \tilde{A}\, :\, \delta\!\rho_2(x)< 0 \mbox{ and } \delta\!\rho_1(x)>-2\sqrt{g}\delta\!\rho_2\mbox{ or }-\sqrt{g}\delta\!\rho_2(x)>2\delta\!\rho_1(x) \rt\},\\
\tilde A_3&=& \{x\in \tilde{A}\, :\, \delta\!\rho_2(x)< 0\mbox{ and }  \delta\!\rho_1(x)/2 \leq -\sqrt{g}\delta\!\rho_2(x)\leq 2\delta\!\rho_1(x)\}.
\end{eqnarray*}
 {\it Step 2.1.} In $\tilde A_1$, we have $|\delta\!\rho_1|+|\delta\!\rho_2|\leq\delta\!\rho_1+\sqrt{g}\delta\!\rho_2$. Hence, by Cauchy-Schwarz, 
 \begin{equation}\label{preuvestab21}
 \|\delta\!\rho\|_{L^1(\tilde{A}_1)}^2 \,\leq \,|\tilde A_1|\int_{\tilde A_1} (\delta\!\rho_1+\sqrt{g}\delta\!\rho_2)^2\, \leq\, C H(\delta\!\rho).
 \end{equation}
  {\it Step 2.2.} In $\tilde A_2$, there holds  $|\delta\!\rho_1|+|\delta\!\rho_2|\leq C |\delta\!\rho_1+\sqrt{g}\delta\!\rho_2|$ and as in the previous step, 
   \begin{equation}\label{preuvestab22}
 \|\delta\!\rho\|_{L^1(\tilde{A}_2)}^2 \, \leq\, C H(\delta\!\rho).
 \end{equation}
 {\it Step.2.3.} In $\tilde A_3$, we deduce from $\delta\!\rho_1\leq-2\sqrt{g}\delta\!\rho_2$ and the condition $\rho_2^0+\delta\!\rho_2\geq0$ that 
 \[
 \delta\!\rho_1\,\leq\, -2\sqrt{g}\delta\!\rho_2\, \leq \, 2\sqrt{g}\max \rho_2^0\, =\, \lambda.
 \]
 Let us note 
 \[
 m\, =\, \int_{\tilde A_3}\delta\!\rho_1\, \leq\, \alpha_1.
 \]
 We have 
 \[
H(\delta\!\rho)\, \geq\,  \int_{\tilde A_3} c_1(|x|) \delta\!\rho_1 \, \geq\, \inf\lt\{ \int_{\tilde A} c_1(|x|) v\,:\, v\in L^1(\tilde A,[0,\lambda]),\ \int_{\tilde{A}} v = m\rt\}. 
 \]
 Since $c_1$ is radial and increasing, the solution of the optimization problem is given by $v_\star=\lambda\mathbf{1}_{B_{r_\star}\setminus B_{r_0}}$ where $r_\star>r_0$ is such that $\int_{\tilde{A}} v_\star=m$. Since $c_1(r)\geq c_1'(r_0)(r-r_0)$ for $r>r_0$ with $c_1'(r_0)>0$, we deduce 
 \[
\int_{\tilde A} c_1(|x|) v_\star(x)\, \geq \, c\, m^2,
 \] 
 for some $c>0$ depending on $\lambda$, $r_0$ and $c_1'(r_0)$. This yields 
 \[
  \lt(\int_{\tilde A_3}\delta\!\rho_1\rt)^2 \, =\, m^2\, \leq\, C  \int_{\tilde A} c_1(|x|) v_\star(x)\, \leq\, C H(\delta\!\rho).
 \]
 Since $|\delta\!\rho_2|\leq2/\sqrt{g}\, \delta\!\rho_1$ in $\tilde A_3$, we conclude that
   \begin{equation}\label{preuvestab23}
 \|\delta\!\rho\|_{L^1(\tilde{A}_3)}^2 \, \leq\, C H(\delta\!\rho).
 \end{equation}
 Gathering~\eqref{preuvestab21},\eqref{preuvestab22},\eqref{preuvestab23}, we get 
 \begin{equation}\label{preuvestab2}
 \|\delta\!\rho\|_{L^1(\tilde{A})}^2 \, \leq\, C H(\delta\!\rho).
 \end{equation}
 
  {\it Step 3. Estimating $\|\delta\!\rho\|_{L^1( B_{r_0})}$.} Proceeding as in Step 2 and exchanging the roles of $c_1$ and $c_2$ and $\delta\!\rho_1$ and $\sqrt{g}\delta\!\rho_2$, we obtain
  \begin{equation}\label{preuvestab3}
 \|\delta\!\rho\|_{L^1(B_{r_0})}^2 \, \leq\, C H(\delta\!\rho).
 \end{equation}
 Eventually, \eqref{preuvestab1},\eqref{preuvestab2},\eqref{preuvestab3} yield the desired estimate~\eqref{preuvestab0}.
\end{proof}

\begin{remark}
 One cannot hope for a stronger inequality of the form
 \[\|\rho-\rho^0\|_{1}\les  E(\rho)-E(\rho^0)\qquad\mbox{or}\qquad \|\rho-\rho^0\|^2_{2}\les   E(\rho)-E(\rho^0).\]
since one can easily get a contradiction by exchanging $\rho_1^{0}$ and $\rho^{0}_2$ on small balls close to $\partial B_{r_0}$.
\end{remark}

\subsection{Approximation of the Thomas-Fermi limit by the Gross-Pitaievskii functional}
We prove that minimizers of $F_\eps$ converge to $(\sqrt{\rho^0_1},\sqrt{\rho^0_2})$ as stated in Theorem~\ref{thmGPtoTF}.
\begin{proof}[Proof of Theorem~\ref{thmGPtoTF}]
 Let $\eta^\eps$ be minimizers of $F_\eps$ for $\eps\in(0,1)$. Regularizing $(\sqrt{\rho^0_1},\sqrt{\rho^0_2})$, one can easily construct a competitor $\tilde\eta_\eps$ with $F_\eps(\tilde\eta_\eps)-{E_0}/{\eps}\le C$. In particular, for minimizers $\eta^\eps$, there holds
\[ \frac{E((\eta^\eps_1)^2,(\eta_2^\eps)^2)-E_0}\eps \le F_\eps(\eta^\eps)-\frac{E_0}{\eps}\le C\]
 so that~\eqref{quantiestim} implies 
 \[
 \|((\eta_1^\eps)^2,(\eta_2^\eps)^2)-(\rho^0_1,\rho^0_2)\|_{L^1} \, \leq\, C \eps^{1/2}.
 \]
 Now using that $|\sqrt{a}-\sqrt{b}|^2\leq |\sqrt{a}-\sqrt{b}||\sqrt{a}+\sqrt{b}|=|a-b|$, we have indeed that 
 \[
  \left\|(\eta_1^\eps,\eta_2^\eps)-\left(\sqrt{\rho^0_1},\sqrt{\rho^0_2}\right)\right\|_{L^2}\, \leq\, \left( \|((\eta_1^\eps)^2,(\eta_2^\eps)^2)-(\rho^0_1,\rho^0_2)\|_{L^1} \right)^{1/2} \, \leq\, C \eps^{1/4}.
 \]

\end{proof}

\section{The crossover case}\label{crossover}
We now study what happens when $g=1+\xi \eps$ for some $\xi>0$. For $\eta$ satisfying~\eqref{constrainteta}, the energy then reads
\[F_\eps(\eta_1,\eta_2)=\eps \int_{\R^2} |\nb \eta_1|^2 + |\nb \eta_2|^2 + \frac1\eps\left[\int_{\R^2}
  \frac{1}2 \eta_1^4   + \frac{1}2 \eta_2^4  +  K \eta^2_1\eta^2_2 +V(\eta_1^2+\eta_2^2)\right] + \frac{\xi}{2} \int_{\R^2} \eta_2^4.  \]
 
   Let us first rewrite the energy in a more convenient way. For this, let $\ov{\eta}_\eps$ be the minimizer of 
 \[G_\eps(\eta)=\eps \int_{\R^2} |\nb \eta|^2 + \frac1\eps  \int_{\R^2} \left[ \frac{1}2 \eta^4  +V\eta^2\right]\]
 under the constraint 
 \[\int_{\R^2} \eta^2 =\overline{\alpha},\]
 where $\overline{\alpha}=\alpha_1+\alpha_2$. It is well known \cite{IgMil, KaSour} that $\oee^2$ converges to
 \[\rhob=(R^2-V(x) )_+\]
 where $R$ is such that $\int_{\R^2} \rhob=1$. We denote by $\mathcal{D}$ the support of $\rhob$ (which is $B_R$ when $V=|x|^2$). We first rewrite the energy in a more convenient form.
 
 \begin{proposition}
  For $u=(u_1,u_2)$ a pair of non-negative functions, let 
  \[\wtilde{F}_\eps(u)= \eps \int_{\R^2}  \oee^2 (|\nabla u_1|^2+ |\nabla u_2|^2)+\frac{1}{\eps}\int_{\R^2} \left( \frac{1}{2} \oee^4\left(1- ( u_1^2+u_2^2)\right)^2 +(K-1) \oee^4 u_1^2u_2^2\right)\]
  then if  $\eta=\oee u$, there holds
  \[
  F_\eps(\eta)= G_\eps(\oee)+\wtilde{F}_\eps(u)+\int_{\R^2} \frac{\xi}{2} \oee^4 u_2^4
  \]

 \end{proposition}

 \begin{proof}
 The proof follows as in \cite{AftRoyo}. We use the Lassoued-Mironescu trick \cite{LaMi} and write $\eta_1=\ov{\eta}_\eps u_1$, $\eta_2=\ov{\eta}_\eps u_2$ to get 
 \be\label{LassMiro1}
 |\nabla \eta_1|^2+|\nabla \eta_2|^2=(u_1^2+u_2^2)|\nabla \oee|^2 +\oee \nabla \oee \cdot \nabla (u_1^2+u_2^2) +\oee^2 (|\nabla u_1|^2+ |\nabla u_2|^2)
 \ee
 The function $\oee$ solves the Euler-Lagrange equation
 \be\label{ELTF}
 -\eps \Delta \oee +\frac{1}{\eps}(  \oee^3+ V(x) \oee)=\lambda_\eps \oee
 \ee
 where $\lambda_\eps$ is some constant. Multiplying the equation~\eqref{ELTF} by $\oee (u_1^2+u_2^2)$, integrating and using integration by parts we get
 \[\eps \int_{\R^2} (u_1^2+u_2^2)|\nabla \oee|^2 +\oee \nabla \oee\cdot \nabla (u_1^2+u_2^2)=\lambda_\eps \overline{\alpha}-\frac{1}{\eps} \int_{\R^2} \oee^4(u_1^2+u_2^2) +V \oee^2(u_1^2+u_2^2).\]
 On the other hand, multiplying~\eqref{ELTF} by $\oee$ and integrating, we find
 \[\lambda_\eps \overline{\alpha}=G_\eps(\oee)+\frac{1}{2\eps}\int_{\R^2} \oee^4\]
 so that~\eqref{LassMiro1} leads to
 \[\eps \int_{\R^2} |\nabla \eta_1|^2+|\nabla \eta_2|^2=\int_{\R^2} \eps \oee^2 (|\nabla u_1|^2+ |\nabla u_2|^2) +\frac{1}{\eps} \left( \frac{1}{2} \oee^4(1- 2( u_1^2+u_2^2)) -V\oee^2 (u_1^2+u_2^2)\right)+G_\eps(\oee)\]
and therefore
 \begin{align*}
  F_\eps(\eta_1,\eta_2)=\,& G_\eps(\oee)+ \int_{\R^2}  \frac{\xi}{2} \oee^4 u_1^4+ \eps \oee^2 (|\nabla u_1|^2+ |\nabla u_2|^2) \\
  & \qquad +\frac{1}{\eps} \left( \frac{1}{2} \oee^4\left(1- 2( u_1^2+u_2^2) +2u_1^2u_2^2+(u_1^4+u_2)^4\right) +(K-1) \oee^4 u_1^2u_2^2\right)\\
  =\,& G_\eps(\oee)+ \int_{\R^2}  \frac{\xi}{2} \oee^4 u_1^4+ \eps \int_{\R^2} \oee^2 (|\nabla u_1|^2+ |\nabla u_2|^2)\\
  &\qquad +\frac{1}{\eps} \int_{\R^2}\left( \frac{1}{2} \oee^4\left(1- ( u_1^2+u_2^2)\right)^2 +(K-1) \oee^4 u_1^2u_2^2\right)\\
  =\,& G_\eps(\oee)+\wtilde{F}_\eps(u_1,u_2)+\int_{\R^2}  \frac{\xi}{2} \oee^4 u_2^4,
 \end{align*}
which completes the proof.
 \end{proof}

For $E$ a set of locally finite perimeter in $\mathcal{D}$ (see \cite{AmbFusPa,Maggibook}), let
\[\mathcal{F}(E)=  \int_{\partial E} \rhob^{3/2} \qquad \textrm{and} \qquad \mathcal{V}(E)=\int_E \rhob.\]
It is proved in \cite[Th. 1.1]{GolRoyo} (see also \cite{AftRoyo}) that for all $p<\oo$, $\wtilde{F}_\eps$ $L^p-\Gamma$ converges to the functional
\[
\mathcal{G}_0(u)=
\begin{cases} 
\ds \sigma_{K} \mathcal{F}(E)  \quad & \textrm{if } u_1=\chi_E, u_2=\chi_{E^c} \textrm{ and } \mathcal{V}(E)=\alpha_1,\\
                       + \oo & \textrm{otherwise.}
                       \end{cases}
                       \]
    where $\sigma_{K}>0$ is defined by the one dimensional optimal transition problem
\[\sigma_{K}=\inf \left\{ \int_{\R} | \eta'_1|^2 + | \eta'_2|^2+  \frac{1}{2}\left(\eta_1^2+\eta_2^2-1\right)^2 +(K-1) \eta_1^2\eta_2^2
: \lim_{-\infty} \eta_1=0, \ \lim_{+\infty} \eta_1=1\right\}.
\]  Since $ \frac{\xi}{2}\int  \oee^4 u_2^4$ is a continuous perturbation of    $\wtilde{F}_\eps(u_1,u_2)$, we immediately obtain the following result.
     \begin{theorem}
      For every $p<\oo$, the functional $\wtilde{F}_\eps(u_1,u_2)+ \frac{\xi}{2}\int  \oee^4 u_2^4$, $L^p-\Gamma$ converges to 
      \[
      \Gxi(u)=
      \begin{cases} \ds \sigma_{K} \mathcal{F}(E)+\frac\xi2\int_{E^c} \rhob^2  \quad & \textrm{if } u_1=\chi_E, u_2=\chi_{E^c} \textrm{ and } \mathcal{V}(E)=\alpha_1,\\
                       + \oo & \textrm{otherwise.}
       \end{cases}
       \]
     \end{theorem}
Up to dividing $\Gxi$ by $\sigma_{K}$ and modifying $\xi$, we can assume that $\sigma_{K}=1$. If $ u_1=\chi_E, u_2=\chi_{E^c}$, we will, by a slight abuse of notation denote $\Gxi(u)$ by $\Gxi(E)$. We now want to study the minimizers of $\Gxi$ (whose existence follows from the Direct Method). As in \cite{FigMag}, by making a spherical symmetrization, we can 
restrict the analysis to spherically symmetric sets. For a given set $E$ and a given half line $\ell$ starting from zero, such symmetrization is defined by replacing for every $r>0$ the spherical slice $E\cap \partial B_r$ by the spherical cap 
$K(E,r)$ centered in $\ell\cap \partial B_r$ and such that $\HH^1(K(E,r))=\HH^1(E\cap \partial B_r)$.
\begin{proposition}\label{regulsym}
 For every minimizer $E_\xi$ of $\Gxi$, $\partial E_{\xi} \cap \mathcal{D}$ (recall that $\mathcal{D}$ is the support of $\rhob$) is a $C^{\infty}$ hypersurface. Moreover, for every half line $\ell$ starting from zero, there exists a minimizer $E_\xi$ such that for every $r>0$, $\partial B_r\cap E_{\xi}$ is an arc of circle centered in $\ell$. 
\end{proposition}
\begin{proof}
 Let $E_\xi$ be a minimizer of $\Gxi$. The regularity of $\partial E_{\xi}$ is a consequence of the regularity theory for minimal surfaces.
 Indeed, since $\rhob$ is locally bounded away from zero in $\mathcal{D}$, any minimizer of $\Gxi$ is locally a quasi-minimizer of the perimeter (see \cite[Th. 3.2]{FigMag} or \cite{Maggibook, DuzStef} for instance). From this, one can infer $C^{1,\alpha}$ regularity of $E_{\xi}$. Since $\rhob$ is smooth in $\mathcal{D}$ further regularity follows.
The symmetry follows as in \cite[Th. 3.2]{FigMag} since spherical symmetrization reduces $\mathcal{F}$ and leaves $\int_{E_\xi^c} \rhob^2$ 
 invariant. \end{proof}

When $\xi$ is small, the perimeter term is dominant. In this case, we are in a situation similar to the one studied in \cite{AftRoyo,GolRoyo}.
\begin{proposition}
 There exists $\alpha_0\in [0,1/2)$ such that for every $\alpha_1\in (\alpha_0,\overline{\alpha}/2-\alpha_0]$, there exists $\xi_0(\alpha_1)$ such that for every $\xi\le \xi_0(\alpha_1)$, 
 the  minimizer of $\Gxi$ is not radially symmetric. 
\end{proposition}

\begin{proof}
 It is proved in \cite{AftRoyo} that for every such $ \alpha_1$, the minimizer of $\mathcal{F}$ under volume constraint is not radially symmetric and thus there exists $E_{\alpha_1}$ with 
 \[\mathcal{F}(E_{\alpha_1})<\inf_{\mathcal{V}(F)=\alpha_1 }\left\{ \mathcal{F}(F) \ : \ F \textrm{ radially symmetric }\right\}.\]
 but since 
 \[\Gxi(E_{\alpha_1})=\mathcal{F}(E_{\alpha_1})+\xi \int_{E_{\alpha_1}^c} \rhob^2\]
 it is clear that 
 \[\Gxi(E_{\alpha_1})<\inf_{\mathcal{V}(F)=\alpha_1 }\{ \mathcal{F}(F) \ : \ F \textrm{ radially symmetric } \} \ < \ \inf\{ \Gxi(F) \ : \ F \textrm{ radially symmetric } \}\]
 for $\xi$ small enough.
\end{proof}

We now study the situation of large $\xi$. Our main result is a rigidity result stating that for large (but not infinite) $\xi$, the unique minimizer is the centered ball.

\begin{theorem}\label{theostab}
 For every $\alpha_1\in (0,\overline{\alpha})$, there exists $\xi_1(\alpha)$ such that for every $\xi\ge \xi_1(\alpha)$, the unique minimizer of $\Gxi$ is the centered ball $\tilde B$ such that $\mathcal{V}(\tilde B)=\alpha_1$.
\end{theorem}
In the rest of the section, $\alpha_1$ is fixed. In order to ease notation, we assume that the unit ball is such that $\mathcal{V}(B)=\alpha_1$ (the general case follows by dilation). Theorem \ref{theostab} follows from a combination of two results.
The first one is a stability result for the volume term $\int_{E^c} \rhob^2$:

\begin{proposition}\label{propstab}
There exists $c>0$ such that for every set  $E$ with $\mathcal{V}(E)= \mathcal{V}(B)$,
\be \label{quantvol}
\int_{E^c} \rhob^2-\int_{B^c} \rhob^2\,\ge\, c \left(\int_{E^c\Delta B^c} \rhob \right)^2
\ee
\end{proposition}

\begin{proof}
  Let $E$ be such that $\mathcal{V}(E)= \mathcal{V}(B)$ then 
 \[\int_{E^c} \rhob^2-\int_{B^c} \rhob^2=\int_{E^c\cap B} \rhob (1-|x|^2)+\int_{ B^c\cap E} \rhob (|x|^2- 1)\]
Let 
\[\overline{V}=\int_{E^c\cap B} \rhob \]
 and $\delta$ be such that 
 \[\int_{B\backslash B_{1 -\delta}} \rhob =\overline{V}.\]
 Letting $F_\delta=B\backslash B_{1 -\delta}$, since $\int_{(E^c\cap B)\backslash F_\delta}\rhob=\int_{F_\delta\backslash E^c} \rhob$,  we have
 \[
 \int_{E^c\cap B} \rhob (1-|x|^2) -\int_{F_\delta} \rhob (1-|x|^2)\ge
 \lt(\inf_{(E^c\cap B)\backslash F_\delta} (1-|x|^2) -\sup_{F_\delta \backslash E^c}  (1-|x|^2) \rt)\int_{F_\delta\backslash  E^c} \rhob \ge 0 
 \]
and therefore, $F_\delta$ minimizes $\int_G \rhob (1-|x|^2)$ among sets $G\subset B$ with $\int_G \rhob=\overline{V}$ so that
\begin{equation}\label{estim1}
\int_{E^c\cap B} \rhob (1-|x|^2)\,\ge\, \int_{F_\delta} \rhob(1-|x|^2) \, \ge \, c\overline{V}^2
\end{equation}
 Similarly, letting $F_{\tilde{\delta}}=B_{1+ \tilde{\delta}}\cap B^c$ with $\int_{F_{\tilde \delta}} \rhob =\overline{V}=\int_{ B^c\cap E} \rhob$,
 then $F_{\tilde \delta}$ minimizes $\int_G \rhob(|x|^2-1)$ among $G\subset B^c$ with $\int_G \rhob=\overline{V}$ and thus
 \[\int_{ B^c\cap E} \rhob (|x|^2- 1)\ge \int_{F_{\tilde{\delta}}} \rhob (|x|^2- 1)\ge C \overline{V}^2.\]
 Together with~\eqref{estim1}, this gives~\eqref{quantvol}.
 \end{proof}
The second is an estimate on the possible instability of the ball for $\mathcal{F}$.
\begin{proposition}\label{propestiminstab}
 There exist $\eps>0$  and $C>0$ such that for  every set $E$ with  $\mathcal{V}(E)= \mathcal{V}(B)$  and $\int_{E\Delta B} \rhob\le \eps$
\be\label{stabilitytrue}\mathcal{F}(E)-\mathcal{F}(B)\,\ge\, -C\left(\int_{E\Delta B} \rhob \rt)^2. 
\ee
\end{proposition}
 Since the proof of Proposition \ref{propestiminstab} is long and involved, we postpone it. Let us show first how Proposition \ref{propstab} and Proposition \ref{propestiminstab} yield together Theorem \ref{theostab}.
 Let $E_\xi$ be a minimizer of $\Gxi$ then  using $B$ as competitor, we obtain thanks to~\eqref{quantvol} and~\eqref{stabilitytrue},
\begin{multline*}
  \lt(\int_{E_\xi\Delta B} \rhob\rt)^2\, \le\, \dfrac1c  \lt(\int_{E_\xi^c}\rhob^2-\int_{B^c} \rhob^2\rt)
  \, =\, \dfrac1{c\xi} \lt( \lt[\mathcal{G}_\xi(E_\xi)-\mathcal{F}(E_\xi)\rt] -  \lt[\mathcal{G}_\xi(B)-\mathcal{F}(B)\rt] \rt)   \\
  \leq\,  \frac{1}{c\xi} \lt(\mathcal{F}(B)-\mathcal{F}(E_\xi)\rt)\,\le\, \frac{C}{c\xi}\lt(\int_{E_\xi\Delta B} \rhob\rt)^2.
 \end{multline*}
This implies $\lt(\int_{E_\xi\Delta B} \rhob\rt)^2=0$ for $\xi> C/c$ and concludes the proof of Theorem \ref{theostab}. \\

Before going into the proof of Proposition \ref{propestiminstab}, let us comment a bit on the statement and give the strategy for proving it. As explained below, in general, we do not expect the ball
to be a local minimizer of $\mathcal{F}$ (and in particular, we cannot expect~\eqref{stabilitytrue} to be true with a plus sign on the right-hand side). However,~\eqref{stabilitytrue} shows that in some sense, the Hessian of $\mathcal{F}$ at the ball is bounded from below. 
The proof of~\eqref{stabilitytrue} follows the strategy of \cite{CicLeo} for proving the quantitative isoperimetric inequality (see also \cite{AcFuMo, FigMag}). Inequality~\eqref{stabilitytrue} is first shown for nearly spherical sets, borrowing ideas from \cite{fuglede}. 
The proof is then finished by arguing by contradiction, constructing a sequence $(E_n)$ converging to the ball and contradicting the inequality. Using a Selection Principle and regularity theory for minimal surfaces, it is possible to replace $(E_n)$ 
by a better sequence $(F_n)$, still contradicting~\eqref{stabilitytrue} but converging in a much stronger way to the ball. Since for $n$ large enough $F_n$ are nearly spherical, we reach a contradiction and Proposition \ref{propestiminstab} is proved.\\

We thus start by proving~\eqref{stabilitytrue} for nearly spherical sets (in the sense of Fuglede \cite{fuglede}).
\begin{proposition}\label{propspherical}
There exist $\eps_0>$  and $C_0>0$ such that if $\eps\le \eps_0$, every set $\partial E=\{ (1+u(x))x \ : \ x\in \partial B\}$ with $\|u\|_{Lip}\le \eps$ and $\mathcal{V}(E)= \mathcal{V}(B)$  satisfies
\[\mathcal{F}(E)-\mathcal{F}(B)\ge -C_0\left(\int_{E\Delta B} \rhob \rt)^2.
\]

\end{proposition}
\begin{proof}
Recall that $\rhob(x)=(R^2-|x|^2)_+$ with $R>1$. The condition $\int_{E} \rhob=\int_B \rhob$ can be written as
\[
\int_{\partial B} \frac{1}{2}R^2 \lt((1+u)^2-1\rt)-\frac{1}{4} \lt((1+u)^4-1\rt)=0
\]
which leads to 
\be\label{condvol}
\int_{\partial B} u =\frac{1}{2}\frac{3-R^2}{R^2-1} \int_{\partial B} u^2 + o\lt(\|u\|_{L^2(\partial B)}^2\rt).
\ee

We can now compute the energy
\begin{multline*}
 \frac{1}{(R^2-1)^{3/2}}\lt(\mathcal{F}(E)-\mathcal{F}(B)\rt)\,=\,\int_{\partial B} \lt(1-\frac{2u+u^2}{R^2-1}\rt)^{3/2}(1+u)\lt(1+\frac{|\nabla u|^2}{(1+u)^2}\rt)^{1/2}-1\\
=\,\int_{\partial B}  u \lt(1-\frac{3}{R^2 -1}\rt)+\frac{3}{2} \frac{4-3R^2}{(R^2-1)^2} u^2 +\frac{1}{2}\|\nabla u\|_{L^2(\partial B)}^2 +o\lt(\|u\|_{L^2(\partial B)}^2+\|\nabla u\|_{L^2(\partial B)}^2\rt)
 \end{multline*}
Using~\eqref{condvol}, this turns into
\begin{align} \nonumber\frac{1}{(R^2-1)^{3/2}}\lt(\mathcal{F}(E)-\mathcal{F}(B)\rt)&\,=\,\frac{1}{2}\int_{\partial B} -\frac{R^2(2+R^2)}{(R^2-1)^2} u^2+\|\nabla u\|_{L^2(\partial B)}^2+o(\|u\|_{L^2(\partial B)}^2+\|\nabla u\|_{L^2(\partial B)}^2)\\
\label{avantestimnorm2}
&\,\geq\, \dfrac13\|\nb u\|_{L^2(\partial B)}^2 - C\|u\|_{L^2(\partial B)}^2.
\end{align}

We now claim that for every $\delta>0$, there exists $\Lambda_\delta>0$ such that 
\be\label{estimnorm2}
 \int_{\partial B} u^2\,\leq\, \delta \int_{\partial B} |\nabla u|^2+ \Lambda_\delta \lt(\int_{\partial B} |u|\rt)^2.
\ee
Indeed, denoting by $\bar u=\frac1{2\pi}\int_{\partial B}u$ and using Sobolev embedding and Young inequality, we compute 
\begin{align*}
\int_{\partial B} u^2&=\,\int_{\partial B} (u-\bar u)u + 2\pi (\bar u)^2\, \leq \,\|u-\bar u\|_{L^\infty(\partial B)}  \|u\|_{L^1(\partial B)}+ (2\pi)^{-1} \|u\|_{L^1(\partial B)}^2\\
& \les\, \|\nb u\|_{L^2(\partial B)} \|u\|_{L^1(\partial B)} +  \|u\|_{L^1(\partial B)}^2\\
&\,\leq\, \delta  \|\nb u\|_{L^2(\partial B)}^2 +\Lambda_\delta \|u\|_{L^1(\partial B)}^2.
\end{align*}
This proves~\eqref{estimnorm2}. Using~\eqref{estimnorm2} in~\eqref{avantestimnorm2}, we obtain,
\[
\mathcal{F}(E)-\mathcal{F}(B)\,\geq\, -C\lt(\int_{\partial B} |u|\rt)^2.
\]
Now, since for nearly spherical sets, there holds
\[\int_{E\Delta B} \rhob \simeq |E\Delta B|\simeq \int_{\partial B} |u|,\] 
we can finally conclude the proof.
\end{proof}
\begin{remark}
 In the proof of Proposition \ref{propspherical}, if we assume that $E$ is centered, that is 
 \[\int_E \rhob(x) x=0,\]
 we could do as in \cite{fuglede,CicLeo}  and decompose $u$ in Fourier series on $\partial B$ to get
 \[u=\sum_k u_k Y_k\]
 with $Y_0=1$ and $Y_1= x\cdot \nu$ for an appropriately chosen $\nu\in \partial B$. So that 
 \[\int_{\partial B} \lt(\frac{R^2}{3}(1+u)^3-\frac{1}{5}(1+u)^5-\left(\frac{R^2}{3}-\frac{1}{5}\right)\right)Y_1= \int_E \rhob(x) x\cdot \nu -  \lt(\frac{R^2}{3}-\frac{1}{5}\rt)\int_{\partial B} x\cdot \nu=0.
 \]
We find thanks to~\eqref{condvol}
 \[u_0=\int_{\partial B} u=O\lt(\|u\|^2_{L^2(\partial B)}\rt) \qquad u_1=\int_{\partial B} u Y_1=O\lt(\|u\|^2_{L^2(\partial B)}\rt).\]
 From this, letting $\mu_k=k^2$ being the $k-$th eigenvalue of the Laplacian on $\partial B$,
 \[
  \int_{\partial B} -\frac{R^2(2+R^2)}{(R^2-1)^2} u^2+\int_{\partial B} |\nabla u|^2=\sum_{k\ge 2} u_k^2 ( \mu_k -\frac{R^2(2+R^2)}{(R^2-1)^2}) +o\lt(\|u\|_{2}^2\rt).
 \]
 Since $\mu_k\ge 4$ for $k\ge 2$, and since 
 \[
 4-\dfrac{R^2(2+R^2)}{(R^2-1)^2}>0\qquad \Longleftrightarrow\qquad \lt[R<\dfrac{1}{\sqrt{3}}(5-\sqrt{13})^{1/2}\ \mbox{ or }\ R> \dfrac{1}{\sqrt{3}}(5+\sqrt{13})^{1/2}\rt],
 \]
 we expect the ball to be unstable for $R\in \lt((5-\sqrt{13})^{1/2}/\sqrt{3},(5+\sqrt{13})^{1/2}/\sqrt{3}\rt) $ and stable otherwise.

\end{remark}

In order to go further, we need an isoperimetric inequality which is, we believe, of independent interest.
\begin{lemma}
  There exists  $c=c(\alpha)>0$ such that for every set $E$ with $\mathcal{V}(E)\leq\alpha/2$, we have
  \be\label{isoper}
  \mathcal{F}(E) \ge c \, \mathcal{V}(E)^{5/6}.\ee
\end{lemma}
\begin{proof}
 Up to a dilation, we assume here that $R=1$ that
 is $\rhob(x)=(1-|x|^2)_+$. Since the estimate is sublinear, it is moreover enough to prove it for connected sets. Let us first show that we can assume that $E\subset (\overline{B}_{1/2})^c$. Indeed, if this is not the case then there exists
 $x\in \partial E \cap B_{1/2}$. Then, there are two possibilities. Either $E\subset B_{3/4}$ and the classical isoperimetric inequality already gives~\eqref{isoper} or there is 
 $y\in \partial E\cap B_{3/4}^c$. Since $\partial E$ is connected it contains a path from $x\in B_{1/2}$ to $y\in B_{3/4}^c$ and $\mathcal{F}(E)$ is bounded from below by $\int_{1/2}^{3/4}(1-s^2)=29/192$.  Again, this gives~\eqref{isoper}.
As in Proposition \ref{regulsym}, we can make a spherical rearrangement and assume further that $E$ is spherically symmetric.\\

We now transform our problem in order to work on a periodic strip. Let $S^1$ be the unit torus and consider the diffeomorphism $ \phi : B_1\backslash B_{1/2}\to S^1\times(0,1/2)$ given by
\[\phi(\theta,r)=\left(\frac{\theta}{2\pi},1-r\right).\]
For $F\subset S^1\times(0,1/2)$, we let
\[\mathcal{V}_1(F)=\int_0^1\int_0^{1/2} \chi_F(x,y) y \ dy dx \qquad \textrm{and} \qquad \mathcal{F}_1(F)= \int_{\partial F} y^{3/2}\]
then, since for $E\subset  B_1\backslash B_{1/2}$,
\[\mathcal{V}(E) \sim \mathcal{V}_1(\phi(E)) \qquad \textrm{and} \qquad \mathcal{F}(E) \sim \mathcal{F}_1(\phi(E)),\]
we are left to prove 
\begin{equation}\label{isoperflat}
 \mathcal{V}_1(F)\les \mathcal{F}_1(F)^{6/5} \qquad \qquad \forall F\subset  S^1\times(0,1/2).
\end{equation}
Notice also that since the sets $E\subset  \overline{B}_1\backslash B_{1/2}$ we started with were spherically symmetric, we can further assume
that for every $\overline{y}\in(0,1/2)$, $F\cap\{y=\overline{y}\}$ is a segment centered in $0\times\{\overline{y}\}$ for instance. We then make a convexification step. We define $\widetilde{F}$ as the smallest set which contains $F$ and which is $y$-convex, that is for very $x\in S^1$ the set $\{y>0\, :\, (x,y)\in \widetilde{F} \}$ is a segment (see Figure~\ref{fig:convex}).
\begin{figure}[ht]
\psfrag{d}{$-1/2$}
\psfrag{e}{$1/2$}
\psfrag{F}{$F$}
\psfrag{x}{$x$}
\psfrag{y}{$y$}
\psfrag{G}{$\widetilde{F}\setminus F$}
\begin{center}
\includegraphics[scale=.7]{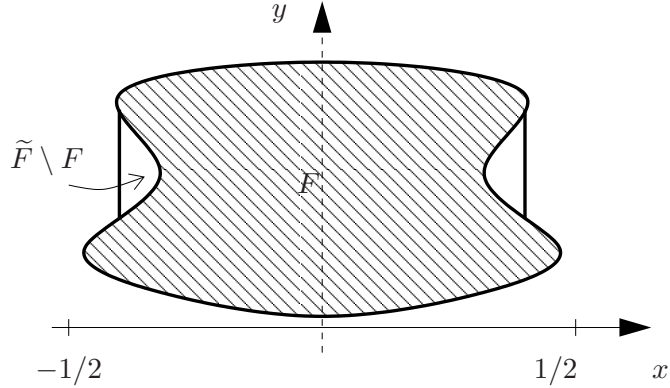}
\end{center}
\caption{\label{fig:convex} Convexification step in the $y$-direction.}
\end{figure}
We have
\[\mathcal{V}_1(\widetilde{F})\ge \mathcal{V}_1(F) \qquad \textrm{and} \qquad \mathcal{F}_1(\widetilde{F})\le \mathcal{F}_1(F).\] 
 Indeed, the first inequality comes from $F\subset \widetilde{F}$ and the second from the fact that for every $x\in S^1$, the shortest (weighted) path between the points $(x,y_1)$ and $(x,y_2)$ is the straight segment.
 It is therefore enough to prove~\eqref{isoperflat} for the sets $\widetilde{F}$. After these two symmetrizations, the set $\widetilde{F}$ is contained between two graphs $y_1:(-T,T]\to (0,1/2)$ and $y_2:(-T,T]\to (0,1/2)$ for some $0<T\le 1/2$ i.e.
 \[\widetilde{F}=\{ (x,y) \, : \,  -T< x\le T, \ y_1(x)\le y\le y_2(x) \}.\]
 Moreover, $y_1$ and $y_2$ are even, $y_1$ is non-decreasing and $y_2$ is non-increasing in $[0,T]$.  Using this parameterization, we have
 \[\mathcal{V}_1(\widetilde{F})=\dfrac12\int_{-T}^T  [y_2^2 -y_1^2] \qquad \textrm{and} \qquad \mathcal{F}_1(\widetilde{F})=\int_{-T}^T y_1^{3/2}\sqrt{1+|Dy_1|^2} +\int_{-T}^T y_2^{3/2}\sqrt{1+|Dy_2|^2},\]
 where for a function $y$ of  locally bounded variation,
 \[
  \mathcal{F}_{1}(y\,)=\,\int_{-T}^T y^{3/2}\sqrt{1+|Dy|^2}\, :\,=\int_{-T}^T y^{3/2}\sqrt{1+|y'|^2}+\sum_{x\in J_y} |y(x^+)^{5/2}-y(x^-)^{5/2}| +\int_{-T}^T  y^{3/2} d|D^c y|,
 \]
see \cite[Th. 5.54]{AmbFusPa}.
 Notice that we can assume that $y_1(\pm T)=y_2(\pm T)$. Indeed,  if $T=1/2$, then $\mathcal{F}_1(\widetilde{F})\ges y_2(1/2)^{3/2}$. In this case, we can add a vertical cut between $y_2(1/2)$ and $y_1(1/2)$. The additional contribution to the weighted perimeter is of order of $y_2(1/2)^{5/2}-y_1(1/2)^{5/2}$ which is controlled by $\mathcal{F}_1(\widetilde{F})$. Eventually, since $\mathcal{F}_1$ is the relaxed functional of its restriction to $C^1$ paths with respect to $L^1$
 convergence \cite[Th. 5.54]{AmbFusPa}, we can further assume that $y_1$ and $y_2$ are smooth.
 After, these symmetrization and regularization steps, we can write,
 \be\label{F1}\mathcal{F}_1(\widetilde{F})\sim \int_{-T}^T y_2^{3/2}+\int_{-T}^T y_1^{3/2}+ y_2(0)^{5/2} -y_1(0)^{5/2}.\ee
 We then consider two cases. First, if $y_1(0)\ll y_2(0)$,  then~\eqref{F1} yields
 \[\mathcal{F}_1(\widetilde{F})\ges \int_{-T}^T y_2^{3/2}+ y_2(0)^{5/2}\]
 from which we deduce
 \[\mathcal{V}_1(\widetilde{F})\ges \int_{-T}^T y_2^2\le y_2(0)^{1/2}\int_{-T}^T y_2^{3/2}\les \mathcal{F}_1(\widetilde{F})^{1/5} \mathcal{F}_1(\widetilde{F})\les \mathcal{F}_1(\widetilde{F})^{6/5}.\]
 Then, in the case $y_1(0)\sim y_2(0)$, by convexity $y_2(0)^{5/2} -y_1(0)^{5/2}\ge \frac{5}{2} y_1(0)^{3/2} (y_2(0)-y_1(0))$ so that~\eqref{F1} becomes
 \[\mathcal{F}_1(\widetilde{F})\ges T y_2^{3/2}(0) + y_2(0)^{3/2} (y_2(0)-y_1(0)).\]
 We infer that
 \begin{multline*}
  \mathcal{V}_1(\widetilde{F})\,=\,\int_{-T}^T (y_2-y_1)(y_2+y_1)
  \, \leq\, 4T y_2(0) (y_2(0)-y_1(0))\\
   \le\, 4T y_2(0) y_2(0)^{4/5} (y_2(0)-y_1(0))^{1/5}\, =\, 4T y_2(0)^{3/2} \left( y_2(0)^{3/2} (y_2(0)-y_1(0))\right)^{1/5}\\
  \,\les\, \mathcal{F}_1(\widetilde{F})\mathcal{F}_1(\widetilde{F})^{1/5}\les \mathcal{F}_1(\widetilde{F})^{6/5},
 \end{multline*}
which concludes the proof of~\eqref{isoperflat}.

 \end{proof} 
 \begin{remark}
  The exponent $5/6$ in~\eqref{isoper} can be easily seen to be optimal by considering as competitor a small ball touching the boundary of $B$. 
 \end{remark}


We can now prove the following $\eps$-regularity result:
\begin{proposition}\label{epsreg}
 Let $\Lambda>0$. Then, there exists $\eps_1>0$ such that if $E$ is a $\Lambda$-minimizer of $\mathcal{F}$, i.e. for every $G$,
 \[ \mathcal{F}(E)\le \mathcal{F}(G)+\Lambda \int_{E\Delta G} \rhob\]
 and if $\int_{E\Delta B} \rhob \le \eps_1$ then $E$ is nearly spherical i.e. $\partial E=\{ (1+u(x)) x \ : \ x\in \partial B\}$ and $\|u\|_{C^{1,\alpha}}\le \eps_0$ (where $\eps_0$ is the one defined in Proposition \ref{propspherical}). 
\end{proposition}

\begin{proof}
 Fix $\delta>0$ then inside $B_{1+\delta}$ by classical $\eps$-regularity results for quasi-minimizers of the perimeter (see \cite[Th. 6.1]{DuzStef} or \cite[Th. II.6.3]{Maggibook} for instance as well as \cite[Lem. 3.6]{CicLeo}), if $\int_{(E\cap B_{1+\delta})\Delta B} \rhob \le \eps_1$ then $\partial E\cap B_{1+\delta}$ is a small $C^{1,\alpha}$ perturbation of $\partial B$ in particular, if 
 $\int_{E\Delta B} \rhob \le \eps_1$, $\partial E\cap B_{1+\delta}\subset B_{1+\frac{\delta}{2}}$ and $E$ can be written as $E=E_1\cup E_2$ where $E_1$ is nearly spherical, $E_2\subset B_{1+\delta}^c$. By testing the $\Lambda-$minimality of $E$ against $E_1$, 
 we find by~\eqref{isoper}
 \[c \mathcal{V}(E_2)^{5/6}\le  \mathcal{F}(E_2) \le \Lambda \mathcal{V}(E_2)\]
 and thus if $\mathcal{V}(E_2)\neq0$,
  \[ \left(\frac{c}{\Lambda}\right)^{6} \le \mathcal{V}(E_2)\le \eps_1\]
 which is absurd for $\eps_1$ small enough.
 \end{proof}
 
 We will also need the following simple lemma which is a weak version of~\eqref{stabilitytrue}.
 \begin{lemma}\label{lemweakstab}
  There exists $\Lambda_1>0$ such that for every set $E$, 
  \be\label{weakstab}
 \F(E)-\F(B)\ge - \Lambda_1 \int_{E\Delta B} \rhob.
  \ee
 \end{lemma}
\begin{proof}
 Let $v$ be a vector field with $\supp v\subset B_R$, $|v|\le 1$, $v=x$ on $\partial B$ and $\|\Div v\|_\infty\le C$ and let 
 \[w=\frac{3}{2}\rhob^{-1/2} \nabla \rhob \cdot v +\rhob^{1/2} \Div v,\]
 so that $\Div (\rhob^{3/2} v)=\rhob w$. Then, for every set $E$, denoting by $\nu^E$ the outward normal to $E$,
 \begin{align*}
  \F(E)-\F(B)&\ge \int_{\partial E} \rhob^{3/2} v\cdot\nu^E -\int_{\partial B}\rhob^{3/2} v\cdot\nu^B\\
  &=\int_{E} \rhob w-\int_B \rhob w\ge -\|w\|_\infty \int_{E\Delta B} \rhob.
 \end{align*}

\end{proof}

We can finally prove~\eqref{stabilitytrue}.

\begin{proof}
 We argue by contradiction. Assume that there exists a sequence of measurable sets $(E_n)$ with $\int_{E_n\Delta B} \rhob \to 0$ and 
 \be\label{hypabsurd}\mathcal{F}(E_n) +2C_0\left(\int_{E_n\Delta B} \rhob \rt)^2\le \mathcal{F}(B),   
 \ee
 where $C_0$ is the constant given by Proposition \ref{propspherical}.
Let $\eps_n=\int_{E_n\Delta B} \rhob$ and for $\Lambda_1>0$ given by Lemma \ref{lemweakstab} and $\Lambda_2>0$, let $F_n$ be a minimizer of 
\be\label{minLambda0}
\F(F)+2\Lambda_1 \lt( \lt[\eps_n-\int_{F\Delta B} \rhob \rt]^2+\eps_n\rt)^{1/2} +\Lambda_2 |\mathcal{V}(F)- \mathcal{V}(B)|.
\ee

{\it Step 1.}  We claim that if $\Lambda_2$ is large enough, $\mathcal{V}(F_n)= \mathcal{V}(B)$ for all $n$.\\
To prove this, we follow the approach of \cite{EspFus} (see also \cite{GolNov} for another approach). Assume by contradiction that the claim does not hold, then there exist sequences of positive numbers $\eps'_k\to 0$ and $\Lambda_{2,k}\to +\infty$, and a sequence of measurable sets $G_k\subset \R^2$ such that $G_k$ minimizes 
\be\label{minLambda}\F(G)+2\Lambda_1 \lt( \lt[\eps_k'-\int_{G\Delta B} \rhob \rt]^2+\eps_k'\rt)^{1/2} +\Lambda_{2,k} |\mathcal{V}(G) -\mathcal{V}(B)|\ee
and for instance $\mathcal{V}(G_k) < \mathcal{V}(B)$ (the other case is similar). In order to get a contradiction, we build  a new sequence $\wtilde G_k$ such that $\mathcal{V}(\widetilde G_k)=\mathcal{V}(B)$ and 
\[\F(\widetilde G_k) +2\Lambda_1 \lt( \lt[\eps_k'-\int_{\widetilde G_k\Delta B} \rhob \rt]^2+\eps_k'\rt)^{1/2}\,\les \,|\mathcal{V}(G_k)-\mathcal{V}(B)|.\]
The construction is a bit delicate since we want all the constants to be uniform in $\eps'_k$. If this were not the case, one could have simply used Almgren's construction (see \cite{Maggibook}).\\
First, testing the energy~\eqref{minLambda} with $B$, we find that  $\F(G_k)\le C$ and therefore, up to extraction, $G_k$ converges in $L^1_{loc}(\mathcal{D})$ to some  $ G_\infty$. 
Moreover, since $|\mathcal{V}(G_k)-\mathcal{V}(B)|\le \frac{C}{\Lambda_{2,k}}$,   $G_\infty$ satisfies $\mathcal{V}(G_{\infty})= \mathcal{V}(B)$. The set $G_\infty$, minimizes
\[\F(G)+2\Lambda_1\int_{G\Delta B} \rhob\]
under the constraint $\mathcal{V} (G)= \mathcal{V}(B)$ (this can be seen for instance by a $\Gamma-$convergence argument). Therefore, by~\eqref{weakstab}, $G_\infty=B$. Notice that arguing as in Proposition \ref{regulsym},
we get that the sets $G_k\cap \mathcal{D}$ are $C^{1,\alpha}$. 
Now, let us fix $\delta>0$ and $r\ll 1$, let us choose $\tilde x_0\in \partial B$ and let us set $x_0=(1+\frac{r}{2})\tilde{x}_0$ and $\bar{C}=(1/2)\int_{B\cap B_r(x_0)} \rhob$. Then since $B_{r/2}(x_0)\cap B=\void$, there hold
\be\label{densestim}\int_{G_k\cap B_{r/2}(x_0)} \rhob\le \delta  \qquad \textrm{and } \qquad \int_{G_k\cap B_{r}(x_0)} \rhob>\bar{C} \ee
for $k$ large enough. Let $0<\sigma_k<1/2$ be a sequence to be fixed later and consider the bilipschitz maps:
\[\Phi_k(x_0+x)-x_0=\begin{cases} (1-3\sigma_k) x  & \textrm{if } |x|\le \frac{r}{2},\\
                 x+\sigma_k (1-\frac{r^2}{|x|^2})x & \textrm{if } \frac{r}{2}\le |x|<r,\\
                 x 	& \textrm{if } |x|\ge r,
                \end{cases}\]
and let $\wtilde{G}_k=\Phi_k(E_k)$.\\
{\it Step 1.1.} We first prove that 
\be\label{estimtildeFn1}\F(G_k)-\F(\wtilde G_k)\ge -C \sigma_k \F(G_k).\ee

Following the notation of \cite{EspFus}, we let  for $x\in \partial G_k$, $T_{k,x}(\tau)=\nabla \Phi_k(x)\circ \tau$ for $\tau \in \pi_{k,x}$ (where $\pi_{k,x}$ is the
 tangent space to $\partial G_k$ at $x$) and
\[J_1T_{k,x}=\sqrt{\det(T_{k,x}^\ast\circ T_{k,x})}\]
be the one Jacobian of $T_{k,x}$ so that 
\[\F(\wtilde G_k)=\int_{\partial G_k} \rhob^{3/2}(\Phi_k(x)) J_1 T_{k,x} d\HH^1.\]
In particular, it is proved in \cite{EspFus} that $J_1 T_{k,x}<1$ in $B_{r/2}(x_0)$ and 
\[ J_1 T_{k,x}\le 1+ 5 \sigma_k\]
in $C_r(x_0)=B_{r}(x_0)\backslash B_{r/2}(x_0)$. We can now decompose,
\begin{align*}
 \F(G_k)-\F(\wtilde G_k)&=\int_{\partial G_k \cap B_r(x_0)} \rhob^{3/2} -\int_{\partial G_k \cap B_r(x_0)} \rhob^{3/2}(\Phi_k(x)) J_1 T_{k,x} d\HH^1\\
 &=\int_{\partial G_k \cap C_r(x_0)} \rhob^{3/2} -\int_{\partial G_k \cap C_r(x_0)} \rhob^{3/2}(\Phi_k(x)) J_1 T_{k,x} d\HH^1\\
 & \quad + 
 \int_{\partial G_k \cap  B_{r/2}(x_0)} \rhob^{3/2} -\int_{\partial G_k \cap B_{r/2}(x_0)} \rhob^{3/2}(\Phi_k(x)) J_1 T_{k,x} d\HH^1.
\end{align*}
But since in $B_r(x_0)$, $|\rhob^{3/2}(x) -\rhob^{3/2}(\Phi_k(x))|\le C\rhob^{3/2}(x) \sigma_k$, 
\begin{align*}\int_{\partial G_k \cap  C_r(x_0)} \rhob^{3/2} - \rhob^{3/2}(\Phi_k(x)) J_1 T_{k,x} d\HH^1&= \int_{\partial G_k\cap C_r(x_0)} (1-J_1 T_{k,x})\rhob^{3/2} \\
&\qquad +J_{1} T_{n,1} (\rhob^{3/2}(x) -\rhob^{3/2}(\Phi_k(x))) d\HH^1\\
 &\ge  -C\sigma_k \int_{\partial G_k \cap  C_r(x_0)} \rhob^{3/2}.
\end{align*}
Similarly,
\begin{align*}
  \int_{\partial G_k \cap  B_{r/2}(x_0)} \rhob^{3/2} - \rhob^{3/2}(\Phi_k(x)) J_1 T_{k,x} d\HH^1&=  \int_{\partial G_k\cap B_{r/2}(x_0)} (1-J_1 T_{k,x})\rhob^{3/2} \\
&\qquad +J_{1} T_{n,1} (\rhob^{3/2}(x) -\rhob^{3/2}(\Phi_k(x))) d\HH^1\\
 &\ge  -C\sigma_k \int_{\partial G_k \cap  B_{r/2}(x_0)} \rhob^{3/2}.
\end{align*}

Hence,~\eqref{estimtildeFn1} follows.\\

{\it Step 1.2.} We now prove that for some $\kappa>0$,
\be\label{estimtildeFn2}
 \mathcal{V}(\wtilde G_k)  -\mathcal{V}(G_k) \ge  \kappa \sigma_k. 
\ee

If  $J\Phi_k$ denotes the Jacobian of $\Phi_k$, it is shown in \cite{EspFus} that 
\be\label{estimJac1}
J\Phi_k\le 1+ 8\sigma_k
\ee
and that in $C_r(x_0)$, 
\be\label{estimJac2}
J\Phi_k\ge 1+ c\sigma_k
\ee
for some $c>0$. We can decompose
\begin{align*}
 \int_{\wtilde G_k} \rhob -\int_{G_k} \rhob &= \int_{G_k\cap B_{r}(x_0)} \rhob(\Phi_k(x)) J\Phi_k(x) -\rhob(x)\\
 &= \int_{G_k\cap C_{r}(x_0)} \rhob(\Phi_k(x)) J\Phi_k(x) -\rhob(x) +\int_{G_k\cap B_{r/2}(x_0)} \rhob(\Phi_k(x)) J\Phi_k(x) -\rhob(x).
\end{align*}
Thanks to~\eqref{estimJac1} and~\eqref{estimJac2}, there holds
\begin{align*}
 \int_{G_k\cap C_{r}(x_0)} \rhob(\Phi_k(x)) J\Phi_k(x) -\rhob(x)&=\, \int_{G_k\cap C_{r}(x_0)} (\rhob(\Phi_k(x))  -\rhob(x))J\Phi_k(x) + (J\Phi_k-1)\rhob\\
 &\ge\, c \sigma_k \int_{G_k\cap C_r(x_0)} \rhob  -2 \int_{G_k\cap C_r(x_0)}  |\rhob(\Phi_k(x))  -\rhob(x))|\\
 &=\,\int_{G_k\cap C_r(x_0)} \rhob \lt(c \sigma_k -2 \lt|\frac{\rhob(\Phi_k(x))-\rhob(x)}{\rhob(x)}\rt| \rt).
\end{align*}
Since $\rhob(x)=(R^2-|x|^2)_+$, we see from the definition of $\Phi_k$ that in $C_{r}(x_0)$, 
\[\lt|\frac{\rhob(\Phi_k(x))-\rhob(x)}{\rhob(x)}\rt|\le  C\sigma_k r\]
hence we can choose $r$ small enough so that $2 C\sigma_k r\le c \sigma_k/2$. So that
\be\label{estimvol1}
\int_{G_k\cap C_{r}(x_0)} \rhob(\Phi_k(x)) J\Phi_k(x) -\rhob(x)\ge \kappa_1 \sigma_k \int_{G_k\cap C_r(x_0)} \rhob,
\ee
for some $\kappa_1>0$. Similarly, thanks to~\eqref{estimJac2},
\begin{align*}
\int_{G_k\cap B_{r/2}(x_0)} \rhob(\Phi_k(x)) J\Phi_k(x) -\rhob(x)&= \int_{G_k\cap B_{r/2}(x_0)} (\rhob(\Phi_k(x))  -\rhob(x))J\Phi_k(x) + (J\Phi_k-1)\rhob\\
&\ge -\kappa_2 \sigma_k \int_{G_k\cap B_{r/2}} \rhob, 
\end{align*}
for some $\kappa_2>0$.  Combining this with~\eqref{estimvol1} and~\eqref{densestim} gives
\[ \mathcal{V}(\wtilde G_k)  -\mathcal{V}(G_k) \ge \sigma_k \lt( \kappa_1 \int_{G_k\cap B_r(x_0)} \rhob- \kappa_2 \int_{G_k\cap B_{r/2}(x_0)} \rhob\rt)\ge   (\kappa_1 \overline{C}-\kappa_2 \delta)\sigma_k,\]
so that~\eqref{estimtildeFn2} holds if $\delta$ is small enough.\\

{\it Step 1.3.} 
 Since $\mathcal{V}(G_k)< \mathcal{V}(B)$ and $\mathcal{V}(G_k)\to \mathcal{V}(B)$,~\eqref{estimtildeFn2} and the continuity of the map $\sigma_k\to \mathcal{V}(\wtilde G_k)$, show that we can find $\sigma_k\to 0$ such that $\mathcal{V}(\wtilde G_k)= \mathcal{V}(B)$. From this we get a contradiction. Indeed, by minimality of $G_k$, 
\begin{multline*}
  \F(G_k)+2\Lambda_1 \lt( \lt[\int_{G_k\Delta B} \rhob -\eps_k'\rt]^2+\eps_k'\rt)^{1/2} +\Lambda_{2,k} |\mathcal{V}(G_k)- \mathcal{V}(B)|\\
  \le\,\F(\wtilde G_k)+2\Lambda_1 \lt( \lt[\int_{\wtilde G_k\Delta B} \rhob -\eps_k'\rt]^2+\eps_k'\rt)^{1/2}\\
  \le\, \F(G_k) +C\sigma_k    +2\Lambda_1 \lt( \lt[\int_{\wtilde G_k\Delta B} \rhob -\eps_k'\rt]^2+\eps_k'\rt)^{1/2},
\end{multline*}
where in the last line we used~\eqref{estimtildeFn1}.  Since  the function $x\to ((x-\eps_k')^2+\eps_k')^{1/2}$ is $1-$Lipschitz, we obtain, recalling~\eqref{estimtildeFn2},
\[ \Lambda_{2,k} \sigma_k  \, \les\,  \sigma_k + \lt|\int_{\wtilde G_k\Delta B} \rhob- \int_{G_k\Delta B} \rhob\rt|.\]
Arguing as for~\eqref{estimtildeFn2}, we can prove that  
\[\lt|\int_{\wtilde G_k\Delta B} \rhob- \int_{G_k\Delta B} \rhob\rt|\,\les\, \sigma_k,\]
from which  we obtain
\[ \Lambda_{2,k} \sigma_k   \les \sigma_k.\]
This is not possible since $\Lambda_{2,k}\to+\infty$.\\

{\it Step 2.} Going back to  $F_n$, minimizers of~\eqref{minLambda0}, we have that $F_n$ converge to some $F_\infty$ minimizing
\[\F(F)+2\Lambda_1 \int_{F\Delta B} \rhob+\Lambda_2 | \mathcal{V}(F)- \mathcal{V}(B)|\]
that is by~\eqref{weakstab}, $F_{\infty}=B$. The $F_n$ are $\Lambda$-minimizers of $\F$. Indeed,  for every $E$,
\[\F(F_n)\le \F(E) +2\Lambda_1 \left|\int_{E\Delta B} \rhob -\int_{F_n\Delta B} \rhob\rt|  +\Lambda_2 \lt| \int_{F_n} \rhob- \int_E \rhob\rt|\le \F(E)+ (2\Lambda_1+\Lambda_2)\int_{E\Delta F_n} \rhob.\]
 Therefore, by Proposition \ref{epsreg} and Proposition \ref{propspherical}, for $n$ large enough,
\be\label{estimFn}\F(F_n)-\F(B)\ge -C_0 \lt(\int_{F_n\Delta B} \rhob\rt)^2.\ee \\

{\it Step 3.} Let $\gamma_n=\int_{F_n\Delta B} \rhob$. Then, by minimality of $F_n$, there holds (using $E_n$ as a competitor and recalling~\eqref{hypabsurd}),
\[\F(F_n)+ 2\Lambda_1 \lt( (\gamma_n-\eps_n)^2 +\eps_n\rt)^{1/2}\le \F(E_n)+2\Lambda_1\eps_n^{1/2}\le \F(B)+2\Lambda_1\eps_n^{1/2}-2 C_0 \eps_n^2\]
which combined with~\eqref{estimFn} gives
\be\label{estimepsn}2\Lambda_1 \lt( (\gamma_n-\eps_n)^2 +\eps_n\rt)^{1/2}-C_0 \gamma_n^2\le -2 C_0 \eps_n^2+2\Lambda_1\eps_n^{1/2}.\ee
From this we obtain
\[2\Lambda_1 \lt( (\gamma_n-\eps_n)^2 +\eps_n\rt)^{1/2}\le C_0 \gamma_n^2+2\Lambda_1\eps_n^{1/2}.\]
Dividing by $2\Lambda_1$ and taking the square of both sides, we get
\[(\gamma_n-\eps_n)^2 +\eps_n\le \frac{C_0^2}{4\Lambda_1^2} \gamma_n^4 + \frac{C_0}{\Lambda_1} \gamma_n^2\eps_n^{1/2} + \eps_n\]
subtracting $\eps_n$ and dividing by $\gamma_n^2$ we deduce
\[\lt(1-\frac{\eps_n}{\gamma_n}\rt)^2 \le \frac{C_0^2}{4\Lambda_1^2} \gamma_n^2 + 2\frac{C_0}{\Lambda_1}\eps_n^{1/2}\]
hence $1-\frac{\eps_n}{\gamma_n}\to 0$. Going back to~\eqref{estimepsn}, we obtain
\[2\Lambda_1\eps_n^{1/2} \lt( \lt[\eps_n\lt(\frac{\gamma_n}{\eps_n}-1\rt)^2+1\rt]^{1/2}-1\rt)\le -C_0\eps_n^2 \lt(2-\frac{\gamma_n^2}{\eps_n^2}\rt)\]
from which we get a contradiction since for $n$ large enough the left-hand side is positive while the right-hand side is negative. This ends the proof of Proposition~\ref{propestiminstab}.
\end{proof}

\section{The case without confining potential}\label{notrap}
\subsection{The functional}
We finally go back to the situation where $g>1$ but where there is no trapping potential. We thus consider for a fixed bounded open set $\Omega\subset \R^2$
\[ F_{\eps}(\eta)=\eps\int_{\Omega} |\nabla \eta_1|^2 +|\nabla \eta_2|^2 +\frac{1}{\eps}\int_{\Omega}\frac{1}{2} \eta_1^4 +\frac{g}{2} \eta_2^4+K \eta_1^2 \eta_2^2\]
with the constraints
\be\label{volcosfin}\int_{\Omega} \eta_i^2 =\alpha_i.\ee
Letting
\[\gamma= (\alpha_1 +\alpha_2 g^{1/2}),\]
it can be easily seen that minimizers of the Thomas-Fermi energy
\[E(\rhob)=\int_{\Omega}\frac{1}{2} \rhob_1^2 +\frac{g}{2} \rhob^2_2+K \rhob_1 \rhob_2\]
under the mass constraint
\[\int_{\Omega} \rhob_i=\alpha_i\]
are given by
\[\rhob_1=\frac{\gamma}{|\Omega|} \chi_E, \qquad \qquad \rhob_2=\frac{\gamma}{|\Omega| g^{1/2}} \chi_{\Omega \backslash E}\]
for any set $E\subset \Omega$ with $|E|= \frac{|\Omega| \alpha_1}{\gamma}$. The minimal energy is then
\[E_0=E(\rhob)=\frac{\gamma^2}{2|\Omega|} \]
This motivates the change of variables
\[
 \widetilde{\eta}_1= \left(\frac{\gamma}{|\Omega|}\right)^{-1/2} \eta_1, \qquad  \widetilde{\eta}_2= \left(\frac{\gamma}{|\Omega| g^{1/2}}\right)^{-1/2} \eta_2, \qquad \widetilde{x}=\left(\frac\gamma{|\Omega|}\right)^{1/2} x, \qquad \jj_\eps(\widetilde{\eta})=\frac{|\Omega|}{\gamma} F_\eps(\eta)-\frac{\gamma}{2\eps}
\]
yielding
\[\jj_\eps(\widetilde{\eta})=\eps \int_{\widetilde{\Omega}} |\nabla \widetilde{\eta}_1|^2 +g^{-1/2}|\nabla \widetilde{\eta}_2|^2 +\frac{1}{\eps} \int_{\widetilde{\Omega}} \frac{1}{2} \left( \widetilde{\eta_1}^2+  \widetilde{\eta_2}^2 -1\right)^2+ (\widetilde{K}-1)\widetilde{\eta_1}^2\widetilde{\eta_2}^2\] 
where $1<\widetilde{K}= \frac{K}{g^{1/2}}$, under the mass constraint
\[ \int_{\widetilde{\Omega}} \widetilde{\eta}_i^2 =\widetilde{\alpha}_i\]
where $\widetilde{\alpha}_1=\alpha_1$ and $\widetilde{\alpha_2}=g^{1/2} \alpha_2$ so that $\widetilde{\alpha}_1+\widetilde{\alpha_2}=\gamma=|\widetilde{\Omega}|$.
Forgetting the tildas and letting $1>\lambda^2=g^{-1/2}$, we finally obtain that the original minimization problem is equivalent to minimizing

\[\jj_\eps(\eta)= \eps \int_{\Omega}  |\nabla \eta_1|^2 +\lambda^2 |\nabla \eta_2|^2+ \frac{1}{\eps}\int_{\Omega} \frac{1}{2}\left(\eta_1^2+\eta_2^2-1\right)^2 +(K-1) \eta_1^2\eta_2^2 
\]
under the volume constraint~\eqref{volcosfin}.

\subsection{The one dimensional transition problem}
Let us introduce, the following energy defined for $\eta=(\eta_1,\eta_2)\in W^{1.2}_{loc}(\R,\R^2)$, 
\[
E_{\lambda,K}(\eta)\ =\ \int_{\R} |\eta_1'|^2 +\lambda^2 | \eta'_2|^2+ W_K(\eta_1,\eta_2),
\]
where for $s,t\ge 0$, we introduced the potential
\[
W_K(s,t)\ =\ \dfrac12(1-s^2-t^2)^2 + (K-1) s^2 t^2.
\]
We consider the minimization problem, 
\begin{equation}
\label{siglbdK}
\sigma_{\lambda,K}\ =\ \inf \left\{ E_{\lambda,K}(\eta) \ :\ \eta=(\eta_1,\eta_2)\in W^{1,2}_{loc}(\R,\R_+^2),\  \lim_{-\infty} \eta_1=0, \ \lim_{+\infty} \eta_1=1\right\}.
\end{equation}

Let us show that problem~\eqref{siglbdK} admits a minimizer. The result also follows from~\cite[Th. 2.1]{ABCP2015} (see also \cite{AliFu}) with a different proof. Let us point out that uniqueness of the optimal profile has been recently shown \cite{Sour}.
\begin{proposition}
There exist minimizing pairs to $\bsigma$. Every such minimizing pair is smooth. Moreover, the following equipartition of energy holds:
\[| \eta'_1|^2 +\lambda^2 | \eta'_2|^2=W_K(\eta_1,\eta_2).\]
\end{proposition}
\begin{proof}
We establish the existence of a minimizing pair $\eta=(\eta_1,\eta_2)$ by the Direct Method of the calculus of variations. The required compactness and semi-continuity result is stated in Lemma~\ref{lemcomp}.  The smoothness of $\eta$ is the consequence of the Euler-Lagrange equations
\[
\left(\begin{array}{c}-2\eta_1'' \\ -2\lambda^2 \eta_2''\end{array}\right)+\nabla W_K(\eta_1,\eta_2)\ =\ 0,\qquad \mbox{in }\R.
\]
For the  equipartition of energy, we take the dot product of the Euler-Lagrange equations with $(\eta_1',\eta_2')^T$. Integrating, we see that the quantity $-| \eta'_1|^2 -\lambda^2 | \eta'_2|^2 + W_K(\eta_1,\eta_2)$ does not depend on $x$. Using the conditions at infinity, we conclude that $| \eta'_1|^2 +\lambda^2 | \eta'_2|^2 = W_K(\eta_1,\eta_2)$ on $\R$.
\end{proof}

\begin{lemma}
\label{lemcomp} 
Let $(\eta^k)_{k\geq 0}\subset W^{1,2}_{loc}(\R,\R_+^2)$ be such that  there exists $C_0\geq 0$ such that for $k\geq 0$, 
\[
\liminf_{x\to -\infty} \eta^k_1< 1/4, \qquad  \limsup_{x\to+\infty} \eta^k_1>3/4\qquad\mbox{and}\qquad E_{\lambda,K}(\eta^k)\leq C_0.
\]
Then, there exist $\eta=(\eta_1,\eta_2)\in W^{1.2}_{loc}(\R,\R^2_+)$, a subsequence (still denoted by $(\eta^k)$) and a sequence $(z_k)\subset \R$ such that 
\[
\eta^k(\cdot-z_k)\ \to\ \eta\qquad\mbox{uniformly on any bounded subset  of } \R.
\]
Moreover, 
\begin{equation}
\label{condinfinity}
\lim_{x\to -\infty} \eta= (0,1), \qquad  \lim_{x\to+\infty} \eta =(1,0),
\end{equation}
and
\[
 E_{\lambda,K}(\eta)\ \leq\ \liminf_{k\to+\oo} E_{\lambda,K}(\eta^k).
\]
\end{lemma}
To prepare for the proof, we start by noticing that (similarly to the Ginzburg-Landau functional for Type-I superconductors \cite{CoGoOtSe}) 
 the Gross-Pitaievskii energy controls a classical double well potential. This is the key ingredient in order to get compactness for sequences of bounded energy. 
 Let us introduce  the relaxed potential,
\begin{equation}
\label{wK}
{w}_K(s)\ =\ \inf_{t\in\R} W_K(s,t)\ =\ \begin{cases}
\dfrac12 (1-s^2)^2 -\dfrac12 (1-K s^2)^2&\mbox{ if }0\leq s< K^{-1/2},\\ \\
\dfrac12 (1-s^2)^2&\mbox{ if } s\geq K^{-1/2}.
\end{cases}
\end{equation}
The function $w_K$ is a standard double-well potential  (see Figure~\ref{fig:relpot}). 
\begin{figure}[ht]
\psfrag{0}{$0$}
\psfrag{1}{$1$}
\psfrag{m}{\!\!\!\!\!$1/2$}
\psfrag{x}{$s$}
\psfrag{d}{$K^{-1/2}$}
\psfrag{W}{$w_K$}
\begin{center}
\includegraphics[scale=.5]{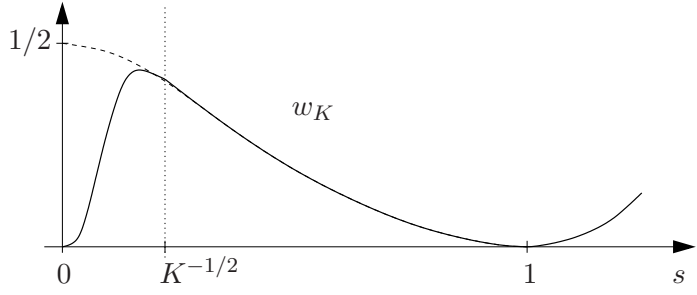}
\end{center}
\caption{\label{fig:relpot} Relaxed potential ${w}_K(s)$ and limit potential $w(s)=(1-s^2)^2$ (dashed line).}
\end{figure}
In particular, since  there holds $w_K(s)=w_K(t)\leq\ W_K(s,t)$, we have for $\eta=(\eta_1,\eta_2)\in W^{1,2}_{loc}(\R,\R_+^2)$
\begin{equation}
\label{dim1}
\int_\R |\eta_1'|^2 + w_K(\eta_1)\ \leq\ E_{\lambda,K}(\eta)\qquad\mbox{and}\qquad \int_\R \lambda^2|\eta_2'|^2 + w_K(\eta_2)\ \leq\ E_{\lambda,K}(\eta).
 \end{equation}
\begin{proof}[Proof of Lemma~\ref{lemcomp}]
Let us consider a sequence $\eta^k=(\eta^k_1,\eta^k_2)$  and $C_0\geq 0$ satisfying the hypotheses of the lemma. Let us fix $k\geq 0$ and let us consider $y<z$ such that either, 
$\eta_1^k(y)=1/4$, $\eta_1^k(z)=3/4$  or $\eta_1^k(y)=3/4$, $\eta_1^k(z)=1/4$. Using the Modica-Mortola trick, we see that
\[
\int_{y}^z |(\eta_1^k)'|^2 + w_K\eta_1^k\ \geq\ 2 \int_{y}^z \sqrt{w_K\eta_1^k}\, |(\eta_1^k)'|\ \geq\ 2\int_{1/4}^{3/4}\sqrt{w_K(s)}\, ds \ =\ \delta.
\]
Since $\delta>0$, taking into account~\eqref{dim1} and the bound $E_{\lambda,K}(\eta^k)\leq C_0$, we deduce that there exists an odd integer $n_k\in [1, C_0/\delta]$ and sequence of intervals
\[
I_0^k=(-\infty,x_1^k),\ I_1^k=(x_1^k,x_2^k),\ \cdots\ ,\ I_{n_k-1}^k=(x_{n_k-1}^k,x_{n_k}^k), \ I_{n_k}^k=(x_{n_k},+\infty),
\]
such that $\eta_1^k(x_j^k)=1/2$ for $j=1,\cdots,n_k$ and there exists $y_0^k,\cdots,y_{n_k}^k$ with $y_j\in I_j^k$ such that 
\[
\begin{cases}
\eta_1^k(y_j^k)=1/4 \quad \mbox{and}\quad  \eta_1^k<3/4\mbox{ in } I_j^k &\mbox{ if $j$ is even,}\\
\eta_1^k(y_j^k)=3/4 \quad \mbox{and}\quad  \eta_1^k>1/4\mbox{ in } I_j^k &\mbox{ if $j$ is odd.}
\end{cases}
\] 
See the example of Figure~\ref{fig:Ijk}. 
\begin{figure}[ht]
\psfrag{0}{{\scriptsize $0$}}
\psfrag{1}{{\scriptsize $1$}}
\psfrag{a}{{\scriptsize $1/4$}}
\psfrag{b}{{\scriptsize $3/4$}}
\psfrag{c}{{\scriptsize $1/2$}}
\psfrag{X}{$x$}
\psfrag{x}{$x_1^k$}
\psfrag{y}{$x_2^k$}
\psfrag{z}{$x_3^k$}
\psfrag{u}{$y_0^k$}
\psfrag{v}{$y_1^k$}
\psfrag{w}{$y_2^k$}
\psfrag{h}{$y_3^k$}
\begin{center}
\includegraphics[scale=.5]{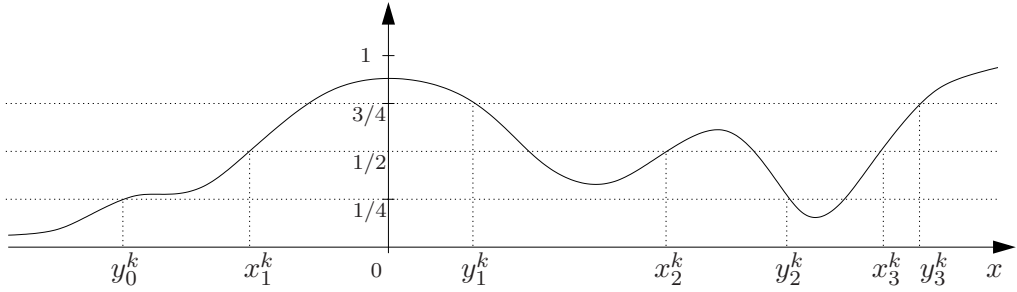}
\end{center}
\caption{\label{fig:Ijk} Example of construction of the intervals $I_j^k$.}
\end{figure}
Now, up to extraction, we assume that $n_k=n$ does not depend on $k$. Now, if the sequence $(x^k_2-x^k_1)_k$ is not bounded, we can extract
a subsequence such that $(x^k_2-x^k_1)\to\oo$. Repeating the process, we can assume that the following property holds true: there exist $R \geq 0$
and a partition of $\{1,..,n\}=\{a_0,\dots,b_0\}\cup\{a_1,\dots,b_1\}\cup\cdots\cup\{a_m,\dots,b_m\}$ with $a_0=1$, $b_m=n$ and $a_{j+1}=b_j +1$ for $j=0,\dots,m-1$, such that,
\[
\forall\,l\in\{0,\cdots,m\}\quad 0<x^k_{b_l}-x^k_{a_{l}}\ \leq\ R,\qquad \forall l\in\{0,\cdots,m-1\}\quad x^k_{a_{l+1}}-x^k_{a_{l}}\ \stackrel{k\to\oo}\longrightarrow\ +\oo.
\]
 Since $a_0=1$ and $b_m=n$ are odd there exists $l_\star\in\{0,\cdots,m\}$ such that $a_{l_\star}$ and $b_{l_\star}$ are odd.  We set $z_k=x_{a_{l_\star}}^k$ and $\tilde{\eta}_1^k=\eta_1^k(\cdot-z_k)$. By construction, there exists a sequence $(R_k)\subset [R,+\infty)$ with $R_k\to+\oo$ such that for $k\geq 0$, 
\begin{equation}
\label{tildxi}
\tilde{\eta}_1^k<3/4\quad \mbox{in } [-R_k,0],\qquad \tilde{\eta}_1^k(0)=1/2\qquad\mbox{and}\qquad \tilde{\eta}_1^k>1/4 \quad\mbox{in } [R,R_k].
\end{equation}
Now, from the energy bound $E_{\lambda,K}(\eta^k)\leq C_0$, we see that $(\tilde{\eta}_1^k)$ is bounded in $W^{1,2}_{loc}(\R)$. Up to extraction, there exists $\eta_1\in W^{1,2}_{loc}(\R)$ such that $\eta_1^k(\cdot-z_k)=\tilde{\eta}_1^k\to\eta_1$ locally uniformly. Moreover, from~\eqref{tildxi}, we have
\begin{equation}
\label{xi}
\limsup_{-\oo}\eta_1 \leq 3/4,\qquad \eta_1(0)=1/2\qquad\mbox{and}\qquad\liminf_{+\oo}\eta_1 \geq 1/4.
\end{equation} 
Similarly, up to extraction, there exists $\eta_2\in W^{1,2}_{loc}(\R)$ such that $\eta_2^k(\cdot-z_k)\to\eta_2$ locally uniformly. By lower semi-continuity of the Dirichlet energy, we have 
\[
E_{\lambda,K}(\eta_1,\eta_2)\ \leq \liminf_{k\to +\infty} E_{\lambda,K}(\eta_1^k,\eta_2^k).
\] 
To end the proof of Lemma~\ref{lemcomp}, we have to establish that $\eta=(\eta_1,\eta_2)$ satisfies the conditions at infinity~\eqref{condinfinity}.
 Since $\int_{\R} |\eta_i'|^2$ is finite $\eta$ admits limits $\eta^\pm$ at $\pm \infty$.   From the bound $\int_\R W_K(\eta_1,\eta_2)\leq C$, we get  $\eta^-,\eta^+\in\{ (0,1) , (1,0) \}$. Eventually~\eqref{xi} implies $\eta^-=(0,1)$, $\eta^+=(1,0)$, that is~\eqref{condinfinity}. 
\end{proof}


\subsection{The $\Gamma-$convergence result}
 In this section we study the $\Gamma-$convergence of $\jj_\eps$ as $\eps$ goes to zero.

\begin{theorem}
 When $\eps\to 0$, $\jj_\eps$ $\Gamma-$converges for the strong $L^1$ topology to 
 \[\J(\eta_1,\eta_2)=\begin{cases}
                                \sigma_{\lambda,K} P(E,\Omega) & \textrm{if } \eta_1=\chi_E=1-\eta_2 \textrm{ and } |E|=\alpha_1\\
                                +\infty & \textrm{otherwise}.
                               \end{cases}\]

\end{theorem}
\begin{proof}
Since $\jj_\eps(\min(\eta_1,1+\delta), \min(\eta_2,1+\delta))\le \jj_\eps(\eta)$ for all $\delta>0$ and since the bound
$\int_{\Omega}\frac{1}{\eps}\left((\eta^\eps_1)^2+(\eta^\eps_2)^2-1\right)^2\le C$ implies that if $\min(\eta^\eps_i,1+\delta)$ (strongly)
converges to some $\eta$ then also $\eta_\eps$ converges to the same $\eta$, we can always assume that sequences which are bounded in energy 
are  bounded in $L^{\infty}$.

Now, we use again that the Gross-Pitaievskii potential controls the double well-potential $w_{K}$ (see~\eqref{wK}). We have, 
\begin{equation}
\label{mainestim}
\int_\Om  w_K(\eta_1) +\int_\Om w_K(\eta_2)\ \leq\ 2\int_\Om W_K(\eta_1,\eta_2).
 \end{equation}

 Thanks to~\eqref{mainestim} and the usual Modica-Mortola argument, we then have that from every sequence $( \eta_1^\eps,\eta_2^\eps)$ of bounded energy, we can extract a subsequence converging strongly in $L^1$ to some pair $(\eta_1,\eta_2)$. 
 Moreover, from the bound on the energy, we get that $\eta_1(x),\eta_2(x) \in \{0,1\}$ and $\eta_1(x)\eta_2(x)=0$ for almost every $x\in\R$. From the volume constraint and the strong convergence, we also deduce that $\int_{\Omega} \eta_1^2=\alpha_1$. 
 The lower bound inequality is then a standard application of the slicing technique (see \cite{braides,AftRoyo,GolRoyo} for instance). The upper bound is also standard. By approximation it is enough doing the construction for a smooth set $E$.
 Let $\delta>0$ be fixed then we can find $T>0$ and $(\eta_1^\delta,\eta_2^\delta)$ with  $\eta^{\delta}_1(-T)=\eta^{\delta}_2(T)=0$ and $\eta^{\delta}_1(T)=\eta^{\delta}_2(-T)=1$ with
\[\int_{-T}^T | (\eta^{\delta}_1)'|^2 +\lambda^2 | (\eta^{\delta}_2)'|^2+ W_K(\eta_1^\delta,\eta_2^\delta)\le \sigma_{\lambda,K} +\delta\]
 Let $d_E$ be the signed distance to $\partial E$.

 We then let
 \[\eta^{\eps}(x)=\eta^\delta\lt(\frac{d_E(x)}{\eps}\rt).\]
 Using the coarea formula it can be seen that $\eta^{\eps}$ converges strongly to $(\chi_E,1-\chi_E)$ and that
 \[\limsup_{\eps\to 0} \jj_\eps(\eta^\eps)\le (\sigma_{\lambda,K}+\delta)P(E,\Omega).\]
 Letting finally 
 \[ \tilde{\eta}^{\eps}=\lt(\frac{\sqrt{\alpha_1}}{\|\eta_1^{\eps}\|_{2}}\,\eta_1^\eps\,,\frac{\sqrt{\alpha_2}}{\|\eta_2^{\eps}\|_{2}}\, \eta_2^\eps\rt),\]
 we have that by definition $\tilde{\eta}^{\eps}$ satisfies the mass constraint. Moreover, using that $\|\eta_i^{\eps}\|_{2}= \alpha_i + O(\eps)$  (and therefore, $\tilde{\eta}^\eps_i=\eta_i(1+ O(\eps))$) we have
  \begin{equation}\label{estimsupdelta}
  \limsup_{\eps\to 0} \jj_\eps(\tilde{\eta}^{\eps})\le (\sigma_{\lambda,K}+\delta)P(E,\Omega).\end{equation}
Since $\delta$ is arbitrary in~\eqref{estimsupdelta}, this concludes the proof of the upper bound.
 \end{proof}

 \begin{remark}
  With a minor adaptation of this proof (see \cite{modicacontact,solci}), one could also deal with Dirichlet boundary conditions (that is impose $\eta_i^\eps=0$ on $\partial \Omega$). 
  Letting
  \[\gamma_i=\inf\left\{ \int_{0}^{+\infty} | \eta'_1|^2 +\lambda^2 | \eta'_2|^2+  \frac{1}{2}\left(\eta_1^2+\eta_2^2-1\right)^2 +(K-1) \eta_1^2\eta_2^2
: \eta_1(0)=\eta_2(0)=0,\  \lim_{+\infty}\eta_i=1\right\},\]
we would obtain as $\Gamma-$limit (at least for $\partial \Omega$ of class $C^2$).

\[\J_{Dir}(\eta_1,\eta_2)=\begin{cases}
                                \sigma_{\lambda,K} P(E,\Omega) +\gamma_1 \mathcal{H}^1(\partial \Omega\cap \partial E)+ \gamma_2 \mathcal{H}^1(\partial \Omega\cap \partial E^c)& \textrm{if } \eta_1=\chi_E=1-\eta_2\\
                                 & \textrm{ and } |E|=\alpha_1\\
                                +\infty & \textrm{otherwise}.
                               \end{cases}\]
                               Notice that in this case, by definition, $\gamma_i\le \sigma_{\lambda,K}+\gamma_j$. Hence, on the macroscopic level, there is always a contact angle. Complete wetting is still possible in the form of an infinitely thin layer of one of the phases (this would show up in the $\gamma_i$). The behavior of $\gamma_i$ has been recently studied in the physics literature \cite{SchayIndcrit}.
 \end{remark}
 \begin{remark}
  When $K\to +\infty$ as $\eps\to 0$ a refinement of the argument giving~\eqref{mainestim} gives also the optimal prefactor (see \cite[Prop. 6.2]{CoGoOtSe}).
 \end{remark}

\subsection{Study of the surface tension}
We study the asymptotics of the surface tension $\sigma_{\lambda,K}$ in the limit regimes $K\to1$ and $K\to+\oo$. We first consider the case $K\to 1$ (weak segregation).
\begin{proposition}\label{propK1}
There holds:
 \[\lim_{K\to 1}\frac{\bsigma}{\sqrt{K-1}}= \frac{2}{3} \frac{1-\lambda^3}{1-\lambda^2}.\]

\end{proposition}

\begin{proof}
Letting $x=(K-1)^{-1/2}y$ in the definition of $\bsigma$, we have
\[\frac{\bsigma}{\sqrt{K-1}}= \inf\left\{\int_{\R} | \eta'_1|^2 +\lambda^2 | \eta'_2|^2+  \eta_1^2\eta_2^2 +\frac{1}{2(K-1)}\left(\eta_1^2+\eta_2^2-1\right)^2
: \lim_{-\infty} \eta_1=0, \ \lim_{+\infty} \eta_1=1\right\}.\]
Hence, using similar $\Gamma-$convergence arguments as above (see also \cite{GolRoyo}), we get that,
\[\lim_{K\to 1}\frac{\bsigma}{\sqrt{K-1}}=\inf\left\{\int_{\R} | \eta'_1|^2 +\lambda^2 | \eta'_2|^2+  \eta_1^2\eta_2^2 
: \lim_{-\infty} \eta_1=0, \ \lim_{+\infty} \eta_1=1, \eta_2^2=(1-\eta_1^2)\right\}.\]
Letting $\eta_1=\cos \phi$ for $\lim_{-\infty} \phi= \frac{\pi}{2}$ and $\lim_{+\infty} \phi =0$, and using that $a^2+b^2\ge 2 ab$, we find
\begin{align*}
 \lim_{K\to 1}\frac{\bsigma}{\sqrt{K-1}}&\ge 2 \int_{-\infty}^\infty |\phi'|( 1-(\cos \phi)^2 (1-\lambda^2))^{1/2} |\cos \phi \sin \phi|\\
 &=2 \left[ \frac{1}{3(1-\lambda^2)} (\cos^2(x)\lambda^2 -\cos^2(x)+1)^{3/2}\right]_0^{\frac{\pi}{2}}\\
 &= \frac{2}{3} \frac{1-\lambda^3}{1-\lambda^2}
\end{align*}
In order to reach equality it is enough to consider the unique solution of
\[\phi' (1-(\cos \phi)^2 (1-\lambda^2))^{1/2} = -\cos \phi \sin \phi,\]
with $\phi(0)=\frac{\pi}{4}$. The solution is indeed unique and defined on $\R$ since $\frac{\cos \phi \sin \phi}{(1-(\cos \phi)^2 (1-\lambda^2))^{1/2}}$ is Lipschitz continuous. Moreover, it is decreasing and has the right values at $\pm \infty$. 
 \end{proof}

\begin{remark}
Notice that the value of $\lim_{K\to 1}\frac{\bsigma}{\sqrt{K-1}}$ exactly coincides with the one found in \cite{barankov} (see also \cite{BVS}). Using~\eqref{mainestim}, it is moreover not hard to prove that there exists $C>0$ (not depending on $\lambda$), such that 
\[C^{-1} \sqrt{K-1}\le \bsigma\le C\sqrt{K-1}\]
when $K\to 1$.
\end{remark}
\begin{remark}
 Proposition \ref{propK1} also holds with the same proof for $\lambda=1$. The limit is then equal to $1$.
\end{remark}

Eventually, we consider the strong segregation asymptotics, $K\to+\oo$ (notice that we recover the same scaling as the one predicted in the physics literature \cite{BVS}). 

\begin{proposition}\label{propKoo}
There exist $0<c\leq C$ such that  for $K>1$ and $\lambda>0$, there holds
\begin{equation}
\label{inegKoo}
\sigma_{\lambda,\oo} -C \lt(\dfrac{\lambda^{1/2}}{(K-1)^{1/4}} + \dfrac{1}{K^{1/2}}\rt) \ \leq \ \bsigma  \ \leq \ \sigma_{\lambda,\oo}-c\left(\frac{1}{K^{1/2}} + \dfrac{\lambda^{1/2}}{K^{1/4}}\rt),
\end{equation}
with 
\begin{align}
\sigma_{\lambda,\oo}&=\inf \left\{ \int_{\R} | \eta'_1|^2 +\lambda^2 | \eta'_2|^2+ \frac{1}{2}\left(\eta_1^2+\eta_2^2-1\right)^2 
: \lim_{-\infty} \eta_1=0, \ \lim_{+\infty} \eta_1=1,\ \eta_1\eta_2\equiv 0\right\},\nonumber\\
&= (1+\lambda)\cfrac {2\sqrt{2}}3 .\label{sigmalambdaoo}
\end{align}
\end{proposition}
\begin{proof}
Let us establish the identity~\eqref{sigmalambdaoo}.
 Since any admissible pair $(\eta_1,\eta_2)$ with finite energy is continuous and satisfies $\eta_1(x)\to1$ as $x\to+\oo$ and $\eta_2(x)\to 1$ as $x\to -\oo$, there exists $x\in \R$ such that $(\eta_1,\eta_2)(x)=0$. Using translation and symmetry, we see that 
\[
\sigma_{\lambda,\oo}\ = \ \gamma_{1} + \gamma_{\lambda},
\]
with 
\[
\gamma_{\lambda}\ =\ \inf \left\{ \int_0^{+\oo}    \lambda^2 | \eta'|^2+ \frac{1}{2}\left(\eta^2-1\right)^2  \ :\ \eta \in W^{1,2}_{loc}([0,+\infty)),\ 
\eta(0)=0, \ \lim_{+\infty} \eta=1\right\}.
\]
The classical Modica-Mortola procedure applies to this minimization problem. For $\lambda>0$  we have (see {\it e.g.} Chapter 6 in~\cite{braidesbook}), 
\begin{equation*} 
 \gamma_\lambda\ =\ 2\lambda^2 \int_0^1\sqrt{\dfrac{(1-s^2)^2}{2\lambda^2}}\, ds \ =\   \frac{2\sqrt{2}}{3} \lambda.
\end{equation*}
Moreover, the minimizer is given by the formula
\begin{equation}
 \label{modmort-m2}
\eta(x)\ =\ \tanh\left( \cfrac {x}{\sqrt{2}\lambda}\right)\qquad \mbox{for }x\geq 0.
\end{equation}
\medskip

Let us now establish the upper bound in~\eqref{inegKoo} (right inequality). If $\lambda \ges K^{-1/2}$ (that is $\frac{\lambda^{1/2}}{K^{1/4}}\ges \frac{1}{K^{1/2}}$), 
inspired by~\eqref{modmort-m2}, we consider the profile
\[
\eta_1(x)\ =\ \begin{cases}
0& \mbox{if }x<0,\\
\tanh\left( \cfrac{x}{\sqrt{2}}\right) & \mbox{if }x\geq 0,
\end{cases}
\qquad
\eta_2(x)\ =\ \begin{cases}
\tanh\left( \cfrac {\delta -x}{\sqrt{2}\lambda}\right)& \mbox{if }x\leq \delta,\\
0& \mbox{if }x>\delta,
\end{cases}
\]
The parameter $\delta>0$  tunes the width of the overlap between the two species (see Figure~\ref{fig:uppbound}). 

\begin{figure}[h]
\psfrag{0}{$0$}
\psfrag{1}{$1$}
\psfrag{x}{$x$}
\psfrag{d}{$\delta$}
\psfrag{e}{$\eta_1$}
\psfrag{f}{$\eta_2$}
\begin{center}
\includegraphics[scale=.5]{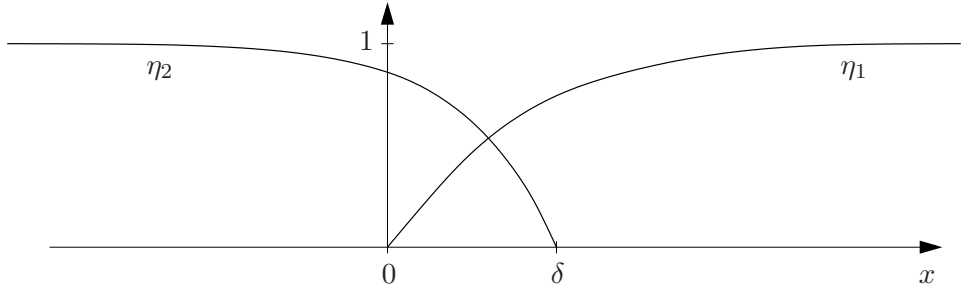}
\end{center}
\caption{\label{fig:uppbound} Profile of the competitor for the upper bound}
\end{figure}
The energy $f(\delta)$ of this competitor is an upper bound for $\sigma_{\lambda,K}$. We compute
\begin{eqnarray*}
f(\delta)& =& E_{\lambda,K}(\eta) \ =\ \int_{\R} | \eta'_1|^2 +\lambda^2 | \eta'_2|^2+  \frac{1}{2}\left(\eta_1^2+\eta_2^2-1\right)^2 +(K-1) \eta_1^2\eta_2^2\\
&=& \int_{\R} | \eta'_1|^2 +\lambda^2 | \eta'_2|^2+\int_{-\oo}^\delta  \frac{1}{2} (1-\eta_2^2)^2 +\int_0^{+\oo}    \frac{1}{2} (1-\eta_1^2)^2 +\int_0^\delta (K \eta_1^2\eta_2^2 -\frac{1}{2})\\
&=& \gamma_1+ \gamma_{\lambda}+\int_0^\delta \left(K \eta_1^2\eta_2^2 -\frac{1}{2}\right)\ =\  \sigma_{\lambda,\oo}+\int_0^\delta \left(K \eta_1^2\eta_2^2 -\frac{1}{2}\right).
\end{eqnarray*}
Using the expressions of $\eta_1$ and $\eta_2$ and the concavity of $t\in\R_+\mapsto \tanh(t)$, we get 
\[
\sigma_{\lambda,K}\ \leq\ f(\delta)\ \leq \  \sigma_{\lambda,\oo} -\frac{\delta}{2} + \dfrac {K\delta^5}{4\lambda^2} \int_0^1 s^2(1-s)^2\, ds  \ = \ \sigma_{\lambda,\oo} -\frac{\delta}{2} + \dfrac {K\delta^5}{120\lambda^2} .
\]
Optimizing in $\delta$, we obtain with $\delta=(12\lambda^2/K)^{1/4}$, 
\[
\sigma_{\lambda,K}\ \leq\ \sigma_{\lambda,\oo} - \dfrac45\left(\dfrac{12\lambda^2}K\right)^{1/4}\ =\  \sigma_{\lambda,\oo} -\dfrac{4 (12)^{1/4}}5\dfrac{\sqrt{\lambda}}{K^{1/4}}.\]
This yields the right inequality of~\eqref{inegKoo} with $c=\dfrac{4 (12)^{1/4}}5$.\\
If on the contrary $\lambda \ll K^{-1/2}$, we keep $\eta_1(x)=\left[\tanh(x/\sqrt{2})\right]_+$ as above and let 
\[\eta_2(x)=\begin{cases}
1  &x\le 0\\
   1-K^{1/2} x &  x\in [0,K^{-1/2}]\\
0  & x\ge K^{-1/2}. 
  \end{cases}
\]

The energy can then be estimated by 
\[E_{\lambda,K}(\eta_1,\eta_2)\le \gamma_1+ \lambda^2 K^{1/2}+\int_0^{K^{-1/2}} \left[\frac12 (1-\eta_2^2)^2+(K-1) \eta_1^2\eta_2^2 -\frac12\right].\]
Notice now that  $\lambda^2 K^{1/2}\ll K^{-1/2} $  and that since  $\tanh(x/\sqrt{2})\le x/\sqrt{2}$,
\[
\int_0^{K^{-1/2}} \left[ (1-\eta_2^2)^2+ 2 (K-1) \eta_1^2\eta_2^2 -1\right] \le \int_0^{K^{-1/2}} \left[  (1-K x^2)^2+ K^2 x^4 -1 \right], \]
 which implies,
\[
\int_0^{K^{-1/2}} \lt[\frac12 (1-\eta_2^2)^2+(K-1) \eta_1^2\eta_2^2 -\frac12\rt] \le \int_0^{K^{-1/2}} \frac{K^2}2 x^4 -K x^2= -\frac7{10} K^{-1/2}.\]
Putting all these estimates together gives the upper bound
\[ \bsigma  \ \leq \ \sigma_{\lambda,\oo}-c K^{-1/2}.\]
\medskip
We now turn to the proof of the lower bound (left inequality of~\eqref{inegKoo}). 

Let us consider an admissible pair $\eta=(\eta_1,\eta_2)\in W^{1,2}_{loc}(\R,\R^2)$ with finite energy. Without loss of generality, we may assume $0<\eta_1(x),\eta_2(x)<1$ for $x\in \R$. By continuity of $\eta$ and from the conditions at infinity, there exists $x_0\in\R$ such that $\eta_1(x_0)=\lambda\eta_2(x_0)=m$ and a maximal interval $I_0=[x_0-\delta_-,x_0+\delta_+]$ such that $m/2 \leq \eta_1,\lambda \eta_2\leq 2m$ in $I_0$. We define $I_-=(-\oo,x_0-\delta_-)$ and $I_+=(x_0+\delta_+,+\oo)$ and split the integration in three parts: $E_{\lambda,K}(\eta)\ =\ E_- + E_0 + E_+$, with
\[
E_i\ =\  \int_{I_i}  |\eta_1'|^2 +\lambda^2 | \eta'_2|^2+ \frac{1}{2}\left(\eta_1^2+\eta_2^2-1\right)^2 +(K-1) \eta_1^2\eta_2^2,\qquad \mbox{for }i\in\{-,0,+\}.
\]
Using the Modica-Mortola trick, we have 
\[
E_+\ \geq\ \int_{x_0+\delta_+}^{+\oo} |\eta_1'|^2 + w_K(\eta_1)\ \geq\ 2\int_{\eta_1(x_0+\delta_+)}^1\sqrt{w_K(s)}\, ds. 
\]
Using $\eta_1(x_0+\delta_+)\leq 2m$ and $2\sqrt{w_K}\geq  \sqrt{2}(1-s^2)-\sqrt{2}\chi_{[0,K^{-1/2}]}$, we obtain,
\begin{equation}
\label{lowbnd1}
E_+\ \geq\ \sqrt2\int_{2m}^1 (1-s^2)\, ds - \sqrt{2} K^{-1/2} \ \geq\ \dfrac{2\sqrt{2}}3 - 2\sqrt{2}m  - \sqrt{2} K^{-1/2}.
\end{equation}
Similarly, 
\begin{equation}
\label{lowbnd2}
E_-\ \geq\ \dfrac{2\sqrt{2}\lambda }3 - 2\sqrt{2} m  - \sqrt{2} K^{-1/2}.
\end{equation}
Now we consider the middle part. Since $\min\left( \int_{I_0}|\eta_1'|, \int_{I_0}\lambda |\eta_2'|\right)\ \geq m/2$, we have by Cauchy-Schwarz inequality, 
\[
\int_{I_0}  |\eta_1'|^2 +\lambda^2 | \eta'_2|^2\ \geq\ \dfrac{m^2}{4(\delta_+-\delta_-)}. 
\] 
On the other hand, since $\eta_1,\lambda\eta_2\geq m/2$ in $I_0$, we have
\[
\int_{I_0}  \frac{1}{2}\left(\eta_1^2+\eta_2^2-1\right)^2 +(K-1) \eta_1^2\eta_2^2\ \geq\ \dfrac{(\delta_+-\delta_-)(K-1)m^4}{16 \lambda^2}.
\]
Optimizing with respect to $(\delta_+-\delta_-)$, we obtain,
\begin{equation}
\label{lowbnd3}
E_0\ =\ \int_{I_0}  |\eta_1'|^2 +\lambda^2 | \eta'_2|^2 + \int_{I_0}  \frac{1}{2}\left(\eta_1^2+\eta_2^2-1\right)^2 +(K-1) \eta_1^2\eta_2^2\ \geq\  \dfrac{\sqrt{K-1}m^3}{4 \lambda}.
\end{equation}
Gathering together~\eqref{lowbnd1},\eqref{lowbnd2},\eqref{lowbnd3}, we get
\[
E_{\lambda,K}(\eta)\ \geq\ \dfrac{2\sqrt{2}}3(1+\lambda) - 4\sqrt{2} m +  \dfrac{\sqrt{K-1}m^3}{4 \lambda}  - 2\sqrt{2} K^{-1/2}.
\]
Minimizing with respect to $m$, the minimum is reached for $m=[16\sqrt{2}\lambda/(3\sqrt{K-1})]^{1/2}$, which yields, 
\[
\sigma_{\lambda,K}\ \geq\ \dfrac{2\sqrt{2}}3(1+\lambda) - \dfrac{32\,2^{3/4}}{3^{3/2}}\dfrac{\lambda^{1/2}}{(K-1)^{1/4}} - \dfrac{2\sqrt{2}}{K^{1/2}}.
\]
This establishes the left inequality of~\eqref{inegKoo} and ends the proof of Proposition~\ref{propKoo}.
\end{proof}

\section*{Acknowledgment}
We are grateful to an anonymous referee which pointed out a great simplification in the proof of Theorem \ref{maintheointro}. We thank G. De Philippis and V. Millot for helpful discussions about quantitative isoperimetric inequalities. B. Merlet was partially supported by the ANR research project GEOMETRYA, ANR-12-BS01-00
14-01

 \bibliography{BEC}

\begin{thebibliography}{10}

\bibitem{AcFuMo}
E.~Acerbi, N.~Fusco, and M.~Morini.
\newblock Minimality via second variation for a nonlocal isoperimetric problem.
\newblock {\em Comm. Math. Phys.}, 322(2):515--557, 2013.

\bibitem{AfLivre}
A.~Aftalion.
\newblock {\em Vortices in {B}ose-{E}instein Condensates}, volume~67 of {\em
  Progress in Nonlinear Differential Equations and Their Applications}.
\newblock Birkh{\"a}user, 2006.

\bibitem{AftNorSour}
A.~Aftalion, B.~Noris, and C.~Sourdis.
\newblock Thomas-{F}ermi approximation for coexisting two component
  {B}ose-{E}instein condensates and nonexistence of vortices for small
  rotation.
\newblock {\em Comm. Math. Phys.}, 336(2):509--579, 2015.

\bibitem{AftRoyo}
A.~Aftalion and J.~Royo-Letelier.
\newblock A minimal interface problem arising from a two component
  {B}ose-{E}instein condensate via {$\Gamma$}-convergence.
\newblock {\em Calc. Var. Partial Differential Equations}, 52(1-2):165--197,
  2015.

\bibitem{AftSour}
A.~Aftalion and C.~Sourdis.
\newblock Interface layer of a two-component {B}ose-{E}instein condensate.
\newblock {\em Preprint}, 2015.

\bibitem{ABCP2015}
S.~Alama, L.~Bronsard, A.~Contreras, and D.~E. Pelinovsky.
\newblock Domain walls in the coupled {G}ross-{P}itaevskii equations.
\newblock {\em Arch. Ration. Mech. Anal.}, 215(2):579--610, 2015.

\bibitem{AliFu}
N.~D. Alikakos and G.~Fusco.
\newblock On the connection problem for potentials with several global minima.
\newblock {\em Indiana Univ. Math. J.}, 57(4):1871--1906, 2008.

\bibitem{AmbFusPa}
L.~Ambrosio, N~Fusco, and D.~Pallara.
\newblock {\em Functions of Bounded Variation and Free Discontinuity Problems}.
\newblock Oxford Mathematical Monographs. Oxford University Press, 2000.

\bibitem{AmbTort}
L.~Ambrosio and V.~M. Tortorelli.
\newblock On the approximation of free discontinuity problems.
\newblock {\em Boll. Un. Mat. Ital. B (7)}, 6(1):105--123, 1992.

\bibitem{Aochui}
P.~Ao and S.~T. Chui.
\newblock Binary {B}ose-{E}instein condensate mixtures in weakly and strongly
  segregated phases.
\newblock {\em Phys. Rev. A}, 58:4836--4840, 1998.

\bibitem{barankov}
R.~A. Barankov.
\newblock Boundary of two mixed {B}ose-{E}instein condensates.
\newblock {\em Phys. Rev. A}, 66:013612, 2002.

\bibitem{BelGolZwi}
P.~Bella, M.~Goldman, and B.~Zwicknagl.
\newblock Study of island formation in epitaxially strained films on unbounded
  domains.
\newblock {\em Accepted for publication in Arch. Ration. Mech. Anal.}

\bibitem{braides}
A~Braides.
\newblock {\em Approximation of Free-Discontinuity Problems}, volume 1694 of
  {\em Lecture Notes in Mathematics}.
\newblock Springer Berlin Heidelberg, 1998.

\bibitem{braidesbook}
A.~Braides.
\newblock {\em {$\Gamma$}-convergence for beginners}, volume~22 of {\em Oxford
  Lecture Series in Mathematics and its Applications}.
\newblock Oxford University Press, Oxford, 2002.

\bibitem{CarFig}
E.~A. Carlen and A.~Figalli.
\newblock Stability for a {GNS} inequality and the log-{HLS} inequality, with
  application to the critical mass {K}eller-{S}egel equation.
\newblock {\em Duke Math. J.}, 162(3):579--625, 2013.

\bibitem{CicLeo}
M.~Cicalese and G.~P. Leonardi.
\newblock A selection principle for the sharp quantitative isoperimetric
  inequality.
\newblock {\em Arch. Ration. Mech. Anal.}, 206(2):617--643, 2012.

\bibitem{CoGoOtSe}
S.~Conti, M.~Goldman, F.~Otto, and S.~Serfaty.
\newblock A branched transport limit of the {G}inzburg-{L}andau functional.
\newblock {\em in preparation}.

\bibitem{DePFraPra}
G.~De~Philippis, G.~Franzina, and Pratelli A.
\newblock Existence of isoperimetric sets with densities ''converging from
  below'' on $\mathbb{R}^n$.
\newblock {\em preprint}, 2014.

\bibitem{DuzStef}
F.~Duzaar and K.~Steffen.
\newblock Optimal interior and boundary regularity for almost minimizers to
  elliptic variational integrals.
\newblock {\em J. Reine Angew. Math.}, 546:73--138, 2002.

\bibitem{EspFus}
L.~Esposito and N.~Fusco.
\newblock A remark on a free interface problem with volume constraint.
\newblock {\em J. Convex Anal.}, 18(2):417--426, 2011.

\bibitem{FigMag}
A.~Figalli and F.~Maggi.
\newblock On the isoperimetric problem for radial log-convex densities.
\newblock {\em Calc. Var. Partial Differential Equations}, 48(3-4):447--489,
  2013.

\bibitem{fuglede}
B.~Fuglede.
\newblock Stability in the isoperimetric problem for convex or nearly spherical
  domains in {${\bf R}^n$}.
\newblock {\em Trans. Amer. Math. Soc.}, 314(2):619--638, 1989.

\bibitem{GolNov}
M.~Goldman and M.~Novaga.
\newblock Volume-constrained minimizers for the prescribed curvature problem in
  periodic media.
\newblock {\em Calc. Var. Partial Differential Equations}, 44(3-4):297--318,
  2012.

\bibitem{GolNovRuf}
M.~Goldman, M.~Novaga, and B.~Ruffini.
\newblock Existence and {S}tability for a {N}on-{L}ocal {I}soperimetric {M}odel
  of {C}harged {L}iquid {D}rops.
\newblock {\em Arch. Ration. Mech. Anal.}, 217(1):1--36, 2015.

\bibitem{GolRoyo}
M.~Goldman and J.~Royo-Letelier.
\newblock Sharp interface limit for two components {B}ose-{E}instein
  condensates.
\newblock {\em ESAIM: COCV}, 2015.

\bibitem{IgMil}
R.~Ignat and V.~Millot.
\newblock The critical velocity for vortex existence in a two-dimensional
  rotating {B}ose-{E}instein condensate.
\newblock {\em J. Funct. Anal.}, 233:260--306, 2006.

\bibitem{KaSour}
G.~D. Karali and C.~Sourdis.
\newblock The ground state of a {G}ross-{P}itaevskii energy with general
  potential in the {T}homas-{F}ermi limit.
\newblock Accepted for publication in Arch. Rational Mech. Anal., 2014.

\bibitem{KaTsuUe}
K.~Kasamatsu, M.~Tsubota, and M.~Ueda.
\newblock Vortices in multicomponent {B}ose-{E}instein condensates.
\newblock {\em Int. J. Mod. Phys. B}, 19(1835), 2005.

\bibitem{KnuMu}
H.~Kn{\"u}pfer and C.~B. Muratov.
\newblock On an isoperimetric problem with a competing nonlocal term {II}:
  {T}he general case.
\newblock {\em Comm. Pure Appl. Math.}, 67(12):1974--1994, 2014.

\bibitem{LaMi}
L.~Lassoued and P.~Mironescu.
\newblock {G}inzburg-{L}andau type energy with discontinuous constraint.
\newblock {\em J. Anal. Math.}, 77:1--26, 1999.

\bibitem{Maggibook}
F.~Maggi.
\newblock {\em Sets of finite perimeter and geometric variational problems},
  volume 135 of {\em Cambridge Studies in Advanced Mathematics}.
\newblock Cambridge University Press, Cambridge, 2012.

\bibitem{MaAf}
P.~Mason and A.~Aftalion.
\newblock Classification of the ground states and topological defects in a
  rotating two-component {B}ose-{E}instein condensate.
\newblock {\em Phys. Rev. A}, 84(3):033611, 2011.

\bibitem{McCarronETall}
D.~J. McCarron, H.~W. Cho, D.~L. Jenkin, M.~P. K{\"o}ppinger, and S.~L.
  Cornish.
\newblock Dual-species {B}ose-{E}instein condensate of $^{87}\mathrm{Rb}$ and
  $^{133}\mathrm{Cs}$.
\newblock {\em Phys. Rev. A}, 84:011603, 2011.

\bibitem{modica}
L.~Modica.
\newblock The gradient theory of phase transitions and the minimal interface
  criterion.
\newblock {\em Arch. Rational Mech. Anal.}, 98(2):123--142, 1987.

\bibitem{modicacontact}
L.~Modica.
\newblock Gradient theory of phase transitions with boundary contact energy.
\newblock {\em Ann. Inst. H. Poincar\'e Anal. Non Lin\'eaire}, 4(5):487--512,
  1987.

\bibitem{MorPra}
F.~Morgan and A.~Pratelli.
\newblock Existence of isoperimetric regions in {$\Bbb{R}^n$} with density.
\newblock {\em Ann. Global Anal. Geom.}, 43(4):331--365, 2013.

\bibitem{OhSten}
P.~{\"O}hberg and S.~Stenholm.
\newblock {Hartree-Fock treatment of the two-component {B}ose-{E}instein
  condensate}.
\newblock {\em Phys. Rev. A}, 57(2):1272--1279, 1998.

\bibitem{PaJILA}
S.~B. Papp, J.~M. Pino, and C.~E. Wieman.
\newblock Tunable miscibility in a dual-species {B}ose-{E}instein condensate.
\newblock {\em Phys. Rev. Lett.}, 101(4):040402, 2008.

\bibitem{solci}
M.~Solci.
\newblock Boundary contact energies for a variational model in phase
  separation.
\newblock {\em J. Convex Anal.}, 13(1):1--26, 2006.

\bibitem{Sour}
C.~Sourdis.
\newblock Remarks on the interface layer of a two-component {B}ose-{E}instein
  condensate.
\newblock {\em Preprint}, 2015.

\bibitem{Timmermans}
E.~Timmermans.
\newblock Phase separation of {B}ose-{E}instein condensates.
\newblock {\em Phys. Rev. Lett.}, 81:5718--5721, 1998.

\bibitem{BVS}
B.~{Van Schaeybroeck}.
\newblock Interface tension of {B}ose-{E}instein condensates.
\newblock {\em Phys. Rev. A}, 78:023624, 2008.

\bibitem{SchayIndcrit}
B.~{Van Schaeybroeck} and J.O. Indekeu.
\newblock Critical wetting, first-order wetting and prewetting phase
  transitions in binary mixtures of {B}ose-{E}instein condensates.
\newblock {\em Phys. Rev. A}, 91(1), 2015.

\end{thebibliography}
 
\end{document}